\documentclass{article}	
\usepackage{geometry} 
\usepackage[utf8]{inputenc}				
\usepackage{lmodern} 
\usepackage{microtype}
\usepackage{xcolor}
\usepackage{psfrag}
\usepackage{amssymb}
\usepackage{appendix}
\usepackage{mathtools}
\usepackage{caption}
\usepackage{subcaption}
\usepackage{amsthm}
\usepackage{amsmath}
\usepackage{fancyhdr}
\usepackage{mathrsfs}

\newtheorem{theorem}{Theorem}
\newtheorem{corollary}{Corollary}
\newtheorem{proposition}{Proposition}
\newtheorem{lemma}{Lemma}

\newtheorem{definition}{Definition}

\newtheorem{remark}{Remark}
\newtheorem{assumption}{Assumption}

\newcommand{\half}{\frac{1}{2}}

\newcommand{\R}{\mathbb{R}}

\newcommand{\bx}{\mathbf{x}}

\newcommand{\VV}{\mathbf{V}}
\newcommand{\VW}{\mathbf{W}}
\newcommand{\VB}{\mathbf{B}}
\newcommand{\VA}{\mathbf{A}}
\newcommand{\VP}{\mathbf{P}}
\newcommand{\VM}{\mathbf{M}}
\newcommand{\Vg}{\mathbf{g}}
\newcommand{\Vu}{\mathbf{u}}

\newcommand{\Vn}{\mathbf{n}}
\newcommand{\Vx}{\mathbf{x}}

\newcommand{\OT}{\mathsf{T}}

\newcommand{\mtsd}{\mathbb{H}^{+\half}(\Gamma)}
\newcommand{\mtsn}{\mathbb{H}^{-\half}(\Gamma)}
\newcommand{\stsd}{H^{+\half}([\Gamma])} 
\newcommand{\stsn}{H^{-\half}([\Gamma])} 
\newcommand{\jspd}{\widetilde{H}^{+\half}([\Gamma])}
\newcommand{\jspn}{\widetilde{H}^{-\half}([\Gamma])} 

\newcommand{\smtp}{\mathbb{X}(\Gamma)}
\newcommand{\stp}{X(\Gamma)}
\newcommand{\sjp}{\widetilde{X}(\Gamma)} 
\newcommand{\smtd}{\mathbb{Y}(\Gamma)}
\newcommand{\std}{Y(\Gamma)}
\newcommand{\sjd}{\widetilde{Y}(\Gamma)} 
\newcommand{\besmtp}{\mathbb{X}_h(\Gamma)}
\newcommand{\besstp}{X_h(\Gamma)}
\newcommand{\besjp}{\widetilde{X}_h(\Gamma)} 
\newcommand{\besmtd}{\mathbb{Y}_h(\Gamma)}
\newcommand{\besstd}{Y_h(\Gamma)}
\newcommand{\besjd}{\widetilde{Y}_h(\Gamma)}

\newcommand{\grad}{\operatorname{\bf grad}}
\renewcommand{\Re}{\operatorname{Re}}

\newcommand{\OV}{\mathsf{V}}
\newcommand{\OW}{\mathsf{W}}
\newcommand{\OB}{\mathsf{B}}

\newcommand{\OA}{\mathsf{A}}
\newcommand{\OM}{\mathsf{M}}
\newcommand{\OD}{\mathsf{D}}
\newcommand{\OR}{\mathsf{R}}
\newcommand{\meshsymb}{\mathcal{T}}

\newcommand{\ba}{\mathsf{a}}
\newcommand{\bb}{\mathsf{b}}
\newcommand{\bm}{\mathsf{m}}
\newcommand{\inter}{\mathcal{I}}

\begin{document}

\title{Calder\'on Preconditioning for Acoustic Scattering at Multi-Screens}

\author{
Kristof Cools\thanks{Ghent University, Ghent, Belgium, e-mail: kristof.cools@ugent.be }
\and Carolina~Urz\'ua-Torres\thanks{TU Delft, Delft Institute of Applied Mathematics, Delft, Netherlands, e-mail: c.a.urzuatorres@tudelft.nl}
  }
\date{}

\maketitle
 
\begin{abstract}
We propose a preconditioner for the Helmholtz exterior problems on multi-screens. For this, we combine quotient-space BEM and operator preconditioning. For a class of multi-screens (which we dub \emph{type A} multi-screens), we show that this approach leads to block diagonal Calder\'on preconditioners and results in a spectral condition number that grows only logarithmically with $h$, just as in the case of simple screens. Since the resulting scheme contains many more DoFs than strictly required, we also present strategies to remove almost all redundancy without significant loss of effectiveness of the preconditioner. 
We verify these findings by providing representative numerical results.

Further numerical experiments suggest that these results can be extended beyond type A multi-screens and that the numerical method introduced here can be applied to essentially all multi-screens encountered by the practitioner, leading to a significantly reduced simulation cost.
\end{abstract}


\section{Introduction} 

We are interested in the scattering of acoustic waves at multi-screens, which 
are geometries composed of essentially two-dimensional piecewise smooth 
surfaces joined together, as shown in Figure~\ref{fig:ms}. Hence, we consider 
the following Dirichlet and Neumann Helmholtz boundary value problems (BVPs) 
in the exterior of the multi-screen $\Gamma\subset\R^3$, with wave number
$\kappa\in\mathbb{C}$, $\Re\kappa\geq 0$, 

\begin{align}
 -\Delta U - \kappa^2 U = 0 \text{ in } \R^3\setminus\overline{\Gamma}, \qquad
 U = g_D  \quad\text{ or } \quad \dfrac{\partial U}{\partial \Vn} = f_N \quad 
 \text{ on }  \Gamma,
\end{align}
plus the \emph{Sommerfield radiation condition}
\begin{equation}
 \lim_{r\to\infty} r \left( \dfrac{\partial U}{\partial r} - i \kappa U \right) 
 = 0, \quad r = \Vert \Vx \Vert,
\end{equation}
where $\Vert \Vx \Vert$ designates the Euclidean norm of a point $\Vx$ in 
$\R^3$, and $g_D$ and $f_N$ are suitable boundary data.

\begin{figure}[!htb]
\centering
\begin{minipage}[c][14em][c]{0.3\textwidth}\centering
\includegraphics[width=0.75\textwidth]{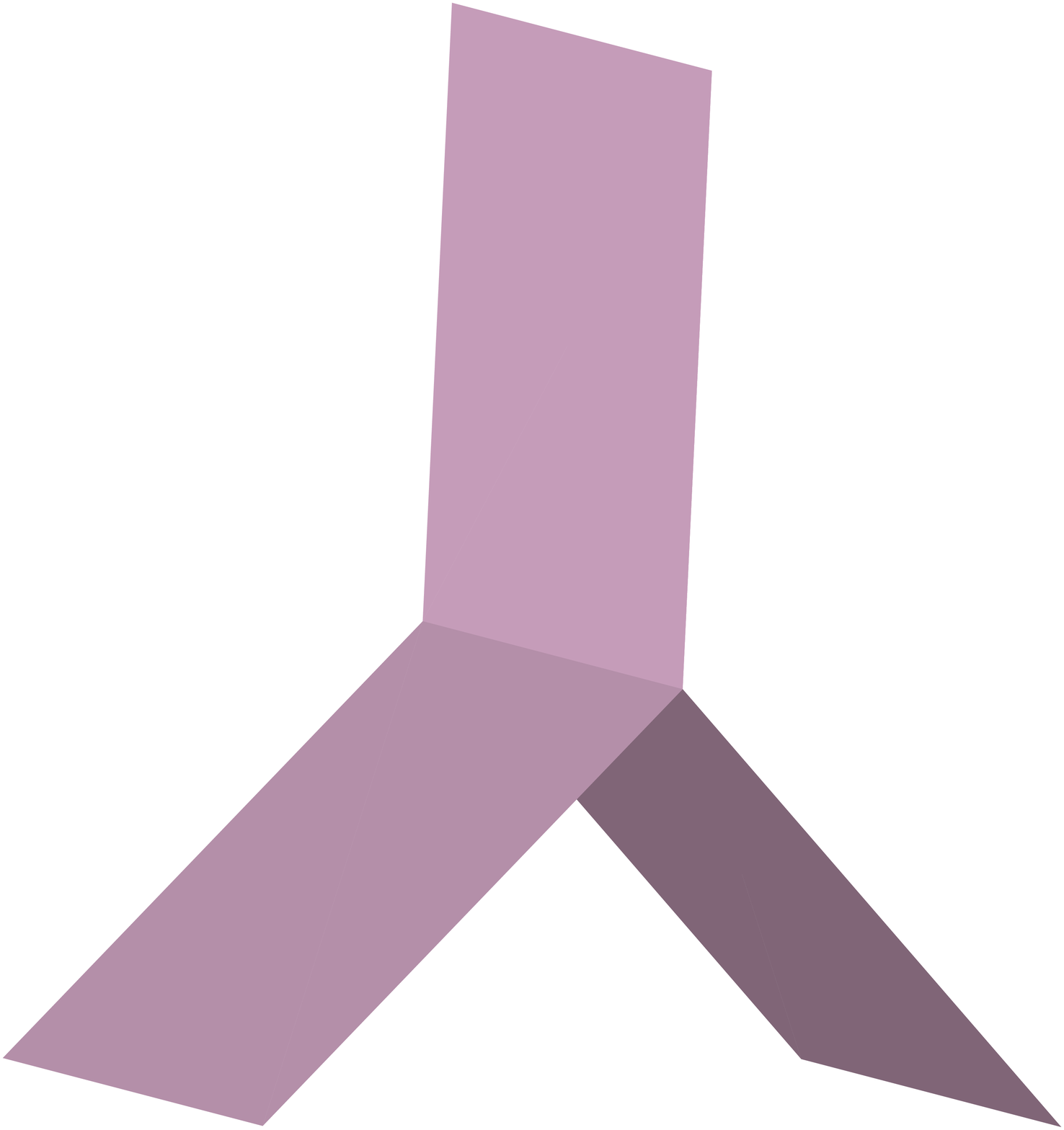}
\end{minipage}
\quad
\begin{minipage}[c][14em][c]{0.3\textwidth}
\includegraphics[width=0.75\textwidth]{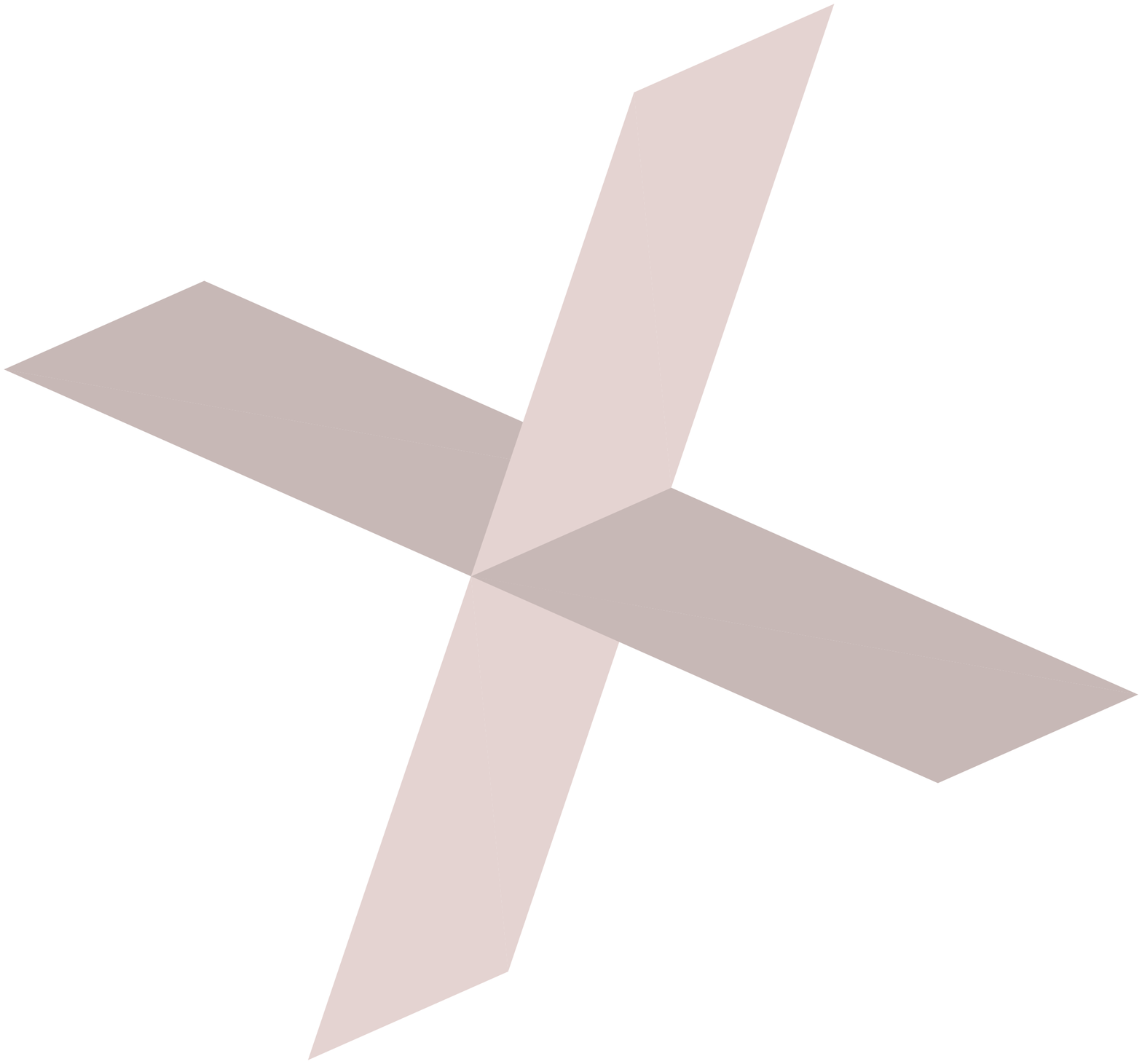}
\end{minipage}%
\caption{Two examples of multi-screen geometries}
\label{fig:ms}
\end{figure}

Our goal is to solve these exterior BVPs efficiently by means of Galerkin 
boundary element methods (BEM) \cite{SaS11} and Calder\'on preconditioning 
\cite{StW98, ChN00}. For this, we recast the BVPs as variational 
\emph{first-kind} boundary integral equations (BIEs) defined for densities on 
the surface of the multi-screen. 

For \emph{simple screens} this approach is well established 
\cite[Section~3.5.3]{SaS11}. Here, we call a simple screen an orientable, 
piecewise smooth two-dimensional manifold with boundary $\mathcal{S}$ embedded in 
$\mathbb{R}^{3}$. For these geometries, the arising variational first-kind BIEs 
are known to be coercive \cite{STE87,ESA90,ERS90} in Sobolev spaces of 
\emph{jumps} of suitable field traces, in $\jspn$ and $\jspd$, 
respectively\cite[Ch.~3]{MCL00}. For these trace spaces, conforming boundary 
element spaces are easily available, and they lead to Galerkin approximations 
and Calder\'on preconditioning whose numerical analysis is well-understood
\cite{McS99, HJU13, HJU20}. 

In contrast, the notion of jumps becomes problematic in multi-screens, since 
they are not globally orientable. For this reason, the tools from simple 
screens cannot be used verbatim on multi-screens. Many alternatives have been 
proposed to tackle this problem. \cite{CTV04, YOS05, Coo12, CoA15a, CoA15b, CoA15c}. It 
is worth pointing out that at the time of writing this article, a rigorous 
analysis of these approaches in suitable trace spaces is not available. 
Furthermore, these approaches lead to ill-conditioned linear systems, yet are 
not amenable to preconditioning. 

Fortunately, recent work by Claeys and Hiptmair offers the mathematical 
framework to overcome these difficulties \cite{ClH13}. The key idea is to see 
trace spaces from the perspective of quotient-spaces and to work with on multi-valued
traces. This new paradigm not only allows for a rigorous analysis, but it also 
paves the way for conforming Galerkin discretization by means of quotient-space 
BEM, as proposed in \cite{CGH21}. 
Indeed, instead of trying to approximate jumps directly, the new approach relies on the Galerkin discretization of multi-trace boundary element spaces. 
With this approach, the related BIEs give rise to Galerkin matrices with large 
null spaces comprised of single-trace functions. Since the right-hand-sides 
of the linear systems of equations are consistent, Krylov subspace iterative 
solvers like GMRES still converge to the right solution. We summarize these 
ideas and results in Section~\ref{sec:QSview}.

Now that the most fundamental issues have been solved, we are in the position to 
investigate how to improve the computational performance of quotient-space BEM 
for multi-screens. 
Indeed, one should note that the arising linear systems are ill-conditioned and 
that the number of GMRES iteration counts increases with mesh refinement. Hence, 
a natural next step -- and the main focus of this paper -- is to devise 
preconditioners for multi-screen problems. In Section~\ref{sec:CaldPreFull}, 
we propose a simple preconditioning strategy based on opposite-order preconditioning, 
also known as Calder\'on preconditioning on closed surfaces. Moreover, we present 
the tools to understand the new preconditioner in the context of operator 
preconditioning. Numerical experiments confirm that this approach reduces considerably 
the number of GMRES iterations required to solve the system.

It is worth mentioning that an advantage of the quotient-space BEM approach is 
that minimal geometrical information is required. However, the disadvantage is 
that one pays with unnecessary computations due to the ``doubling of degrees of
freedom'' underlying the discretization of multi-valued traces. As an 
alternative, we dedicate Section~\ref{sec:CaldPreRed} to discuss reduced 
quotient-space representations that require slightly more geometrical information, 
but minimize computational effort while still rendering efficient Calder\'on preconditioning. Furthermore, we use the tools derived in Section~\ref{sec:CaldPreFull} 
to provide some insight about the requirements that such reductions need to fulfil.

\section{Quotient-Space Perspective} 
\label{sec:QSview}
We briefly summarize the new perspective on trace spaces on multi-screens 
introduced in \cite[Section~4-6]{ClH13} and the quotient-space construction of 
boundary element spaces from \cite{CGH21}.

\subsection{Geometry}
We begin by recalling the rigorous characterization of \emph{multi-screens} as 
given in \cite[Section~2]{ClH13}: 

\begin{definition}[Lipschitz Partition {\cite[Definition 2.2]{ClH13}}]
  \label{def:lp}
  A \textit{Lipschitz partition} of $\mathbb{R}^d$, $d=2,3$, is a finite
  collection of Lipschitz open sets $\left( \Omega_j \right)_{j=0\ldots n}$ 
  such that
  $\mathbb{R}^d = \cup_{j=0}^n \overline{\Omega}_j$ and
  $\Omega_j \cap \Omega_k = \emptyset$, if $j \neq k$.
\end{definition}

\begin{definition}[Multi-screen {\cite[Definition 2.3]{ClH13}}]
  \label{def:ms}
  A \textit{multi-screen} is a subset $\Gamma \subset \mathbb{R}^d$ such that 
  there exists a Lipschitz partition of $\mathbb{R}^d$ denoted
  $\left( \Omega_j \right)_{j=0\ldots n}$ satisfying
  $\Gamma \subset \cup_{j=0}^n \partial{\Omega_j}$ and such that for each
  $j = 0 \ldots n$, we have
  $\overline{\Gamma} \cap \partial \Omega_j = \overline{\Gamma} _j $ where
  $\overline{\Gamma} _j \subset \partial \Omega_j$ is some Lipschitz screen 
  in the sense of Buffa-Christiansen \cite[section 1.1]{BuC03}.
\end{definition}

From a numerical point of view, it will be convenient to classify multi-screens
into three categories. For this, we first need to consider the notion of 
irregular points on the boundary, as in \cite{ClH16}. 

\begin{definition}[Irregular points {\cite[Definition 2.3]{ClH16}}]
Let us consider $\partial \Gamma:= \Gamma \setminus \text{int}(\Gamma)$ 
and introduce the set of regular points of the boundary $\mathcal{P}_R(\partial \Gamma)$ defined as 
$$ \mathcal{P}_R(\partial \Gamma) = \lbrace 
x \in \partial \Gamma \text{ such that } B_x \cap \Gamma = B_x \cap S 
\text{ for some ball }B_x \text{ centred at x and some simple Lipschitz screen }
S \rbrace.$$
We define the set of \textbf{irregular points of the boundary} as $$\mathcal{P}_I(\partial \Gamma)= \partial \Gamma\setminus\mathcal{P}_R(\partial \Gamma).$$
\end{definition}

With this, we can classify our multi-screens as follows:
\begin{itemize}
    \item[\textbf{Type A: }] $\Gamma$ is a multi-screen such that irregular points $\mathcal{P}_I(\partial \Gamma)$ are on the boundary of \emph{all} 
    geometries meeting at the junction line(s).
    \item[\textbf{Type B: }] $\Gamma$ is a multi-screen such that irregular points $\mathcal{P}_I(\partial \Gamma)$ may be in the interior of at least
    one of the geometries meeting at the junction line(s).
    \item[\textbf{Type C: }] $\Gamma$ is a multi-screen without irregular points 
    $\mathcal{P}_I(\partial \Gamma)$.
\end{itemize}

\begin{figure}
    \centering
    \begin{subfigure}[b]{0.3\textwidth}
    \includegraphics[width=\textwidth]{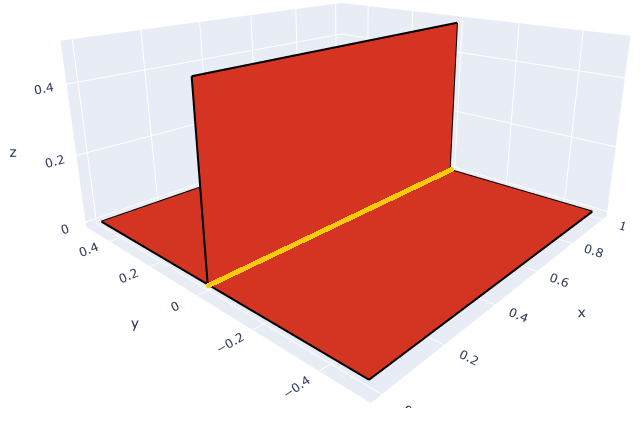}
    \caption{Type A}
    \end{subfigure}
      \begin{subfigure}[b]{0.3\textwidth}
    \includegraphics[width=\textwidth]{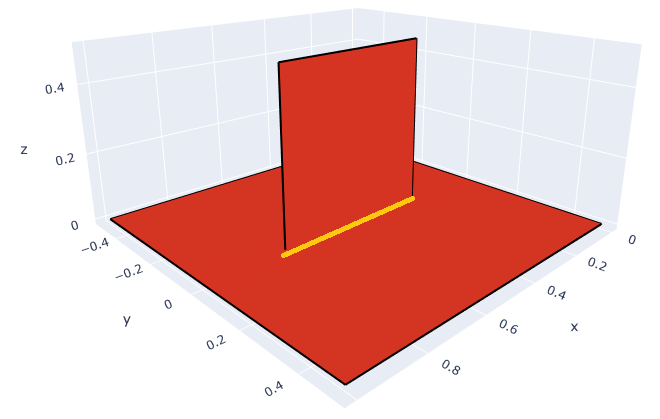}
        \caption{Type B}
        \end{subfigure}
    \begin{subfigure}[b]{0.3\textwidth}
        \includegraphics[width=\textwidth]{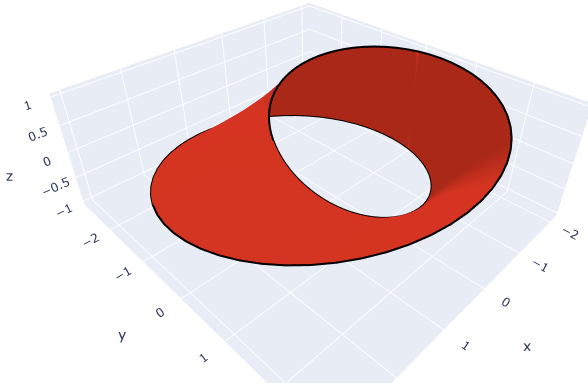}
        \caption{Type C}
    \end{subfigure}
    \caption{Multi-screens can be classified according to the location of their irregular points.}
    \label{fig:geo_types}
\end{figure}

Figure~\ref{fig:geo_types} provides examples of multi-screens in these three 
different classifications. Multi-screens of Type A and Type B arise from applications that we are interested in and will be the focus of this article.

\subsection{Trace Spaces}
\label{ssec:as}
Given a multi-screen $\Gamma\subset \R^d$ , with $d=2,3$, we consider the 
following chains of nested Sobolev spaces\footnote{We refer to \cite[Section~1.1]{GIR86} 
for definitions of the relevant Sobolev spaces.}
\begin{subequations}
  \label{eq:c}
  \begin{gather}
    \label{eq:H1c}
    H^1_{0,\Gamma}(\mathbb{R}^{d})\subset H^1(\mathbb{R}^{d})\subset
    H^1(\mathbb{R}^{d}\backslash\overline{\Gamma}),\\
    \label{eq:Hdivc}
    \mathbf{H}_{0,\Gamma}(\mathrm{div},\mathbb{R}^{d})\subset
    \mathbf{H}(\mathrm{div},\mathbb{R}^{d})\subset
    \mathbf{H}(\mathrm{div},\mathbb{R}^{d}\setminus\overline{\Gamma}),
  \end{gather}
\end{subequations}
where a subscript $X_{0,\Gamma}$ indicates a space obtained as the closure 
in $X$ of smooth functions/vectorfields compactly supported in $\mathbb{R
}^{d}\setminus\overline{\Gamma}$. All inclusions in \eqref{eq:c} define 
closed subspaces, which describe the associated quotient-spaces Hilbert 
spaces. With this, we can define the \textbf{multi-trace spaces} \cite[Section~5]{ClH13}
\begin{subequations}
\label{eq:multi}
\begin{align}
\label{multi1}
\mtsd & := H^1(\mathbb{R}^d\backslash \overline{\Gamma})/
H^1_{0,\Gamma}(\mathbb{R}^d),\\\label{mult2}
\mtsn & := \mathbf{H}(\mathrm{div},\mathbb{R}^d\backslash\overline{\Gamma})/
\mathbf{H}_{0,\Gamma}(\mathrm{div},\mathbb{R}^d).
\end{align}
\end{subequations}
and the \textbf{single-trace spaces}
\cite[Section~6.1]{ClH13}
\begin{subequations}
  \label{eq:sing}
\begin{align}
\label{sing1}
\stsd & := H^1(\mathbb{R}^d)/H^1_{0,\Gamma}(\mathbb{R}^d),\\
\label{eq: sing2}
\stsn & := \mathbf{H}(\mathrm{div},\mathbb{R}^d)/
\mathbf{H}_{0,\Gamma}(\mathrm{div},\mathbb{R}^d).
\end{align}
\end{subequations}

Since the spaces $\stsd$ and $\stsn$ are 
closed subspaces of $\mtsd$ and $\mathbb{H}^{-\half}
(\Gamma)$, respectively \cite[Proposition 6.2]{ClH13}, we can also 
introduce the \textbf{jump spaces} \cite[Section~6.2]{ClH13} as
\begin{gather}
  \label{eq:jumpsp}
  \jspd := \mtsd/\stsd \quad \text{and} \quad \jspn := \mtsn/\stsn.
\end{gather}

\begin{remark}
\label{rem:read}
We note that $H^1(\mathbb{R}^d\backslash \overline{\Gamma})$ and $\mathbf{H}(\mathrm{div},\mathbb{R}^{d}\setminus\overline{\Gamma})$ 
are spaces of functions attaining different values on both sides of $\Gamma$. 
This implies that functions in the multi-trace spaces $\mtsd$ and $\mtsn$ are multi-valued 
on $\Gamma$. In other words, they can take different values on both sides of 
$\Gamma$. One way to grasp this is to imagine an ``infinitesimally inflated'' 
screen, as illustrated in Figure~\ref{fig:2dms} for a 2D multi-screen. With 
this, one can intuitively understand the trace spaces introduced above as follows:
  
\begin{itemize}
\item $\mtsd$ can be seen as a standard Dirichlet trace space on the surface of 
the inflated screen. Similarly, $\mtsn$ can be viewed as the standard space of
Neumann trace space on the surface of the inflated screen.

\item The single-trace space $\stsd$ simply consists of single-valued functions 
on $\Gamma$. One can follow the same intuition for $\stsn$, however, its right 
interpretation as a single-valued normal component requires that one fixes a 
local normal $\mathbf{n}$ on $\Gamma$. 
\end{itemize}

\begin{figure}[!htb]
\centering
\psfrag{G}{$\Gamma$}
\includegraphics[width=0.6\textwidth]{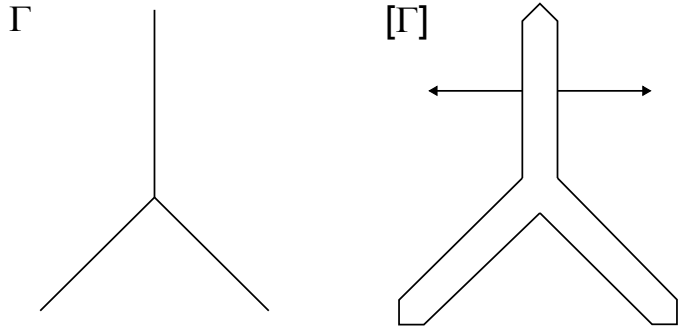}
\caption{Inflating a 2D multi-screen.}
\label{fig:2dms}
\end{figure}

\end{remark}

Next, we consider the canonical surjections
\begin{gather}
\label{eq:cansur}
\pi_D: H^1(\mathbb{R}^d\backslash \overline{\Gamma}) \rightarrow \mtsd \quad
\mathrm{and} \quad \pi_N :  \mathbf{H}(\mathrm{div},
\mathbb{R}^d\backslash \overline{\Gamma}) \rightarrow \mtsn,
\end{gather}
and $H^{1}(\Delta,\mathbb{R}^{d}\setminus\Gamma) := \{v\in
  H^{1}(\mathbb{R}^{d}\setminus\Gamma),\,\Delta v\in L^{2}(\mathbb{R}^{d}
  \setminus\Gamma)\}$.
With this, we are in the position to introduce the relevant trace operators
\begin{align*}
  \text{Dirichlet trace:}\quad & & \gamma_{D}:H^{1}(\mathbb{R}^{d}\setminus
  \Gamma)\to  \mtsd\;, && \gamma_{D} & := \pi_{D} \;,\\
  \text{Neumann trace:}\quad & & \gamma_{N}:H^{1}(\Delta,\mathbb{R}^{d}
  \setminus\Gamma)\to  \mtsn\;, & &\gamma_{N} & := \pi_{N}\circ \grad \;.
\end{align*}
Moreover, we remark that they map onto $\stsd$ and $\stsn$ when restricted to 
${H^1(\mathbb{R}^{d})}$ and ${\mathbf{H}(\mathrm{div},\mathbb{R}^{d})}$, 
respectively.

Finally we introduce a bilinear pairing on
$\mtsd\times \mtsn$:
\begin{gather}
  \label{eq:llgg}
  \ll {u} , {p} \gg := \int\nolimits_{[\Gamma]} {u}{p}\ d\sigma 
  := \int_{\mathbb{R}^d \backslash \overline{\Gamma}} \mathbf{p} \cdot \nabla u 
  + u \mathrm{div}(\mathbf{p}) \ d\bx,
\end{gather}
with $u\in H^1(\R^d\setminus\overline{\Gamma})$ and $\mathbf{p} \in \mathbf{H
  }(\mathrm{div}, \R^d\setminus\overline{\Gamma})$ \cite[Section~5.1]{ClH13}.
Note that this pairing induces the following \emph{isometric dualities} 
\cite[Prop.~5.1 and Section~6.2]{ClH13}
\begin{align}
 \mtsn  \cong \left(\mtsd\right)^\prime, \quad \jspn  \cong \left(\stsd\right)^\prime, 
 \quad \jspd  \cong \left(\stsn\right)^\prime.
\end{align}

The bilinear pairing also offers a characterization of single-trace spaces 
through self-polarity:
\begin{proposition}[{\cite[Proposition 6.3]{ClH13}}]
  \label{prop:polar}
  For ${u} \in \mtsd$ and ${p} \in \mathbb{H}^{
  -\half}(\Gamma)$ the following equivalences hold true:
  \begin{eqnarray*}
& {u} \in \stsd\quad \Longleftrightarrow \quad
  \int_{[\Gamma]}  {u} {q} \ d\sigma = 0 \quad \forall {q}\in 
  \stsn,\\
& {p} \in \stsn \quad \Longleftrightarrow \quad
  \int_{[\Gamma]}  {v} {p} \ d\sigma = 0 \quad \forall {v}\in 
  \stsd.
  \end{eqnarray*}
\end{proposition}

\subsection{Weakly Singular and Hypersingular BIEs}
\label{ssec:sbie}

Let 
$$\mathcal{G}_\kappa(\mathbf{z}) := \dfrac{\exp(\imath\kappa\|\mathbf{z}\|)}
{4 \pi\|\mathbf{z}\|}$$ 
be the radiating fundamental solution of the Helmholtz
equation in $\mathbb{R}^{3}$.
The \emph{weakly singular} boundary 
integral operator (BIO) can be stated in integral form as
\begin{gather}
\label{eq:Vf}
(\OV_{\kappa} {{\phi}})(\mathbf{x}) = \int_{[\Gamma]} \mathcal{G}_{\kappa}
(\mathbf{x}-\mathbf{y}) {{\phi}}(\mathbf{y})\, d\sigma(\mathbf{y})\;,
\quad {{\phi}}\in \mtsn\cap \mathbb{L}^{\infty}(\Gamma),
\end{gather}
where integration is carried out over the virtual inflated screen, \emph{cf.} 
Figure~\ref{fig:2dms}, and $\mathbb{L}^{\infty}(\Gamma)$ is understood as the usual $L^\infty$ space but over the virtual inflated screen.

In order to solve the Dirichlet Helmholtz BVP, we solve the BIE given by
\begin{equation}
  \label{eq:BIEV}
  {{\phi}}\in\mtsn:\quad \OV_{\kappa} ({{\phi}}) = {{g}}_{D},
\end{equation}
which can be written in equivalent variational form as follows:
\begin{equation}
\label{eq:varV}
\mathrm{Find \ } {{\phi}} \in \mtsn \ \mathrm{such \ that} \quad 
\ll \OV_{\kappa} {{\phi}},  {{\psi}} \gg = \ll {{g}_D}, 
{{\psi}} \gg \quad \forall {{\psi}} \in \mtsn.  
\end{equation}

Similarly, solving the Neumann Helmholtz BVP is equivalent to solving the BIE
\begin{gather}
  \label{eq:BIEW}
  {{v}}\in\mtsd:\quad \OW_{\kappa}({{v}}) = {{f}}_{N}.
\end{gather}
Also this BIE can be cast in variational form and this results in the problem: 
\begin{equation}\label{eq:varW} 
\mathrm{Find \ } {{v}} \in  \mtsd \ \mathrm{such \ that} \quad 
\ll  \OW_{\kappa} {{v}}, {{p}} \gg = \ll {{f}_N}, {{p}} \gg 
\quad \forall {{p}} \in \mtsd. 
\end{equation}
 
As shown in \cite[Section~3.3]{SaS11}, the bilinear form on the left-hand side 
of \eqref{eq:varW} can be conveniently expressed by integration by parts over 
the virtual inflated screen for sufficiently regular arguments:
\begin{gather}
  \label{eq:ibp3DW}
  \begin{aligned}
    \ll \OW_{\kappa} {{v}}, {{p}} \gg = \int_{[\Gamma]} \int_{[\Gamma]}
    G_{\kappa}(\mathbf{y}-\mathbf{x}) \Bigl\{&
    (\grad_\Gamma {{v}} \times
    \mathbf{n})(\mathbf{y}) \cdot (\grad_\Gamma {{p}} \times
    \mathbf{n})(\mathbf{x}) \\
    & -\kappa^{2}\mathbf{n}(\mathbf{y})\cdot\mathbf{n}(\mathbf{x})
    {{v}}(\mathbf{y}) {{p}}(\mathbf{x}) \Bigr\}\, d\sigma(\mathbf{y})d\sigma(\mathbf{x}).
  \end{aligned}
\end{gather}

We conclude this section by reminding the reader of some properties of these 
BIEs:
\begin{proposition}[{ \cite[Prop.~8.8]{ClH13}}]
\label{prop:GardIneq}
There exist compact operators
$\mathsf{K}_V:\jspn \rightarrow \stsd$ 
and
$\mathsf{K}_W:\jspd\rightarrow \stsn$ 
such that the following G\aa{}rding inequalities are satisfied

\begin{align}
\label{eq:giv}
\mathrm{Re} \left\lbrace \int_{[\Gamma]} {{q}}(\OV_{\kappa}+\mathsf{K}_{V}) 
\bar{{{q}}} \
d\sigma \right\rbrace &
\geq C_{\OV} \lVert {{q}} \rVert^2_{\jspn}
\quad \forall {{q}} \in \jspn,\\
\label{eq:giw}
\mathrm{Re} \left\lbrace \int_{[\Gamma]} {{v}}(\OW_{\kappa}+\mathsf{K}_{W}) 
\bar{{{v}}} \
d\sigma \right\rbrace &
\geq C_{\OW} \lVert {{v}} \rVert^2_{\jspd}
\quad \forall {{v}} \in \jspd,
\end{align}
with $C_{\OV}, C_{\OW}>0$ depending only on $\kappa$ and $\Gamma$.
\end{proposition}

\begin{lemma}[{\cite[Lemma~3.2]{CGH21}}]
\label{lem:kernVW}
The nullspaces of ${\OV_{\kappa}}$ and ${\OW_{\kappa}}$ agree with ${\stsn}$ 
and $\stsd$, respectively.
\end{lemma}

From these results, we see that ${\OV_{\kappa}}$ and ${\OW_{\kappa}}$ remain 
well-defined on the corresponding jump spaces and are coercive there. Indeed, 
Theorem~8.11 in \cite{steinBook} combined with \cite[Prop.~8.8]{ClH13} 
and \cite[Prop.~8.9]{ClH13} gives us the following inf-sup conditions:

\begin{corollary}
\label{cor:DisInfSupQSpace}
\begin{enumerate}
\item[i)] For a dense sequence of finite dimensional subspaces $(\widetilde{Z}^{
1/2}_h(\Gamma))_{h \in \mathcal{H}}\subset\jspd$ there exist $h_0>0$ such that 
for all $h\leq h_0$ it holds
\begin{align}
\underset{u_h\in\widetilde{Z}^{1/2}_h(\Gamma)}{\inf}\:\underset{v_h\in\widetilde{
Z}^{1/2}_h(\Gamma)}{\sup} \dfrac{\vert\langle \OW_{\kappa} u_h, v_h \rangle\vert}
{\Vert u_h \Vert_{\jspd}\Vert v_h \Vert_{\jspd}} \geq \alpha_{\OW_{\kappa}}>0,
\end{align}
and
\begin{align}
\underset{v_h\in\widetilde{Z}^{1/2}_h(\Gamma)}{\inf}\:\underset{u_h\in\widetilde{
Z}^{1/2}_h(\Gamma)}{\sup} \dfrac{\vert\langle \OW_{\kappa} u_h, v_h \rangle\vert}
{\Vert u_h \Vert_{\jspd}\Vert v_h \Vert_{\jspd}} \geq \alpha_{\OW_{\kappa}}>0,
\end{align}
where $\alpha_{\OW_{\kappa}}$ is independent of $h$.

\item[ii)] For a dense sequence of finite dimensional subspaces $(\widetilde{Z}^{
-1/2}_h(\Gamma))_{h \in \mathcal{H}}\subset\jspn$ there exist $h_0>0$ such that 
for all $h\leq h_0$ it holds
\begin{align}
\underset{\mu_h\in\widetilde{Z}^{-1/2}_h(\Gamma)}{\inf}\:\underset{\varphi_h\in
\widetilde{Z}^{-1/2}_h(\Gamma)}{\sup}\dfrac{\vert\langle \OV_{\kappa} \mu_h, 
\varphi_h\rangle\vert}{\Vert \mu_h \Vert_{\jspn}\Vert \varphi_h\Vert_{\jspn}} 
\geq \alpha_{\OV_{\kappa}}>0,
\end{align}
and
\begin{align}
\underset{\varphi_h\in\widetilde{Z}^{-1/2}_h(\Gamma)}{\inf}\:\underset{\mu_h\in
\widetilde{Z}^{-1/2}_h(\Gamma)}{\sup}\dfrac{\vert\langle \OV_{\kappa} \mu_h, 
\varphi_h\rangle\vert}{\Vert \mu_h \Vert_{\jspn}\Vert \varphi_h\Vert_{\jspn}} 
\geq \alpha_{\OV_{\kappa}}>0,
\end{align}
where $\alpha_{\OV_{\kappa}}$ is independent of $h$.
\end{enumerate}
\end{corollary}

Additionally, these operators are also well-defined on the multi-trace spaces 
$\mtsn$ and $\mtsd$, respectively. However, Lemma~\ref{lem:kernVW} implies that they have non-trivial nullspaces when 
considered on multi-trace spaces. Although this hinders uniqueness of solutions for
\eqref{eq:varV} and \eqref{eq:varW}, {Proposition}~\ref{prop:polar} still 
provides existence, since ${g}_{D}\in\stsd$ and ${f}_{N}\in \stsn$ 
guarantees consistency of the right-hand side linear forms: they vanish on the
single-trace spaces.

\section{Calder\'on Preconditioning for Quotient-Space BEM}
\label{sec:CaldPreFull}

As already mentioned in the introduction, the linear systems arising from 
the discretization of \eqref{eq:varV} and \eqref{eq:varW} using Quotient-space 
BEM are ill-conditioned, which causes that the number of GMRES iteration counts 
increases with mesh refinement. One should note that this is not a particularity 
of Quotient-space BEM. Indeed, we usually encounter this difficulty when using 
low-order BEM discretization of first-kind integral equations on simple screens 
and closed surfaces. In those cases, one typically fixes the problem by using so-called 
Calder\'on preconditioning, which combines Calder\'on identities with operator 
preconditioning to build a very convenient and effective preconditioner \cite{StW98,ChN00,Hip06}.

In this paper, we will extend this approach and devise Calder\'on 
preconditioners for the problem at hand.
Following the policy of operator preconditioning (see, for instance 
\cite{Hip06}), we introduce the following more general notation in order to 
state the results that will hold for both BIEs under consideration (i.e. \eqref{eq:BIEV} and \eqref{eq:BIEW}):
\begin{itemize}
\item Let $\smtp$ and $\smtd$ be \emph{multi-trace} spaces such that 
$\smtd = \left( \smtp \right)^\prime$. 
\item Let $\stp$ and $\std$ be \emph{single-trace} spaces such that 
$\stp \subset \smtp$ and $\std \subset \smtd$. 
\item Let $\sjp$ and $\sjd$ be \emph{jump} spaces such that 
$\sjp = \smtp / \stp$ and $\sjd = \smtd / \std$.  
\end{itemize}

What these spaces will be exactly, depends on whether we are solving the Dirichlet or Neumann problem. For clarity, we will consider each case separately 
in the next Subsections.

Now, in both cases, we are interested in continuous sesquilinear forms $\ba 
\in L(\smtp \times \smtp, \mathbb{C})$, that will characterize the variational 
formulations of our BIEs \eqref{eq:varV} and \eqref{eq:varW}. However, unlike 
in the traditional operator preconditioning setting \cite{Hip06}, we know from 
Corollary~\ref{cor:DisInfSupQSpace} that these sesquilinear forms $\ba$ will 
satisfy an inf-sup condition on $\sjp$, but not on $\smtp$.

Naturally, this will affect the corresponding discrete inf-sup conditions, 
and hence, the condition number bounds. The remainder of this section is 
dedicated to understanding this, and to answering the question of whether the discrete 
inf-sup conditions are satisfied and how they depend on the mesh 
parameter $h$ when using Quotient-space BEM.

Let us begin by introducing the 
notation for the corresponding finite dimensional spaces. On the one hand, we 
will work with
\begin{itemize}
\item $\besmtp \subset \smtp$ : primal \emph{multi-trace} BE space for $\Gamma$; and
\item $\besmtd \subset \smtd$ : dual \emph{multi-trace} BE space for $\Gamma$,
\end{itemize}
which will be actually used for the implementation. We remark that $\besmtp$ 
and $\besmtd$ are Hilbert spaces. On the other hand, we 
consider the finite-dimensional subspaces
\begin{align*}
    \besstp \subset \stp, \qquad \besstd \subset \std, \\
    \besjp \subset \sjp, \qquad \besjd \subset \sjd,
\end{align*}
which will only be used to show our theoretical results. 
It is worth mentioning that we always assume that these finite-dimensional 
subspaces satisfy
\begin{align*}
    \besstp &\subset \besmtp, &\quad \besstd &\subset \besmtd, \\
    \besjp &= \besmtp /\besstp, &\quad \besjd &= \besmtd /\besstd.
\end{align*}

\subsection{Preconditioning the Hypersingular operator}
\label{ssec:precW}

When considering the Neumann problem \eqref{eq:varW}, we will have that $\smtp=\mtsd$, $\stp = \stsd$, $\sjp = \jspd$ for primal 
spaces, and $\smtd=\mtsn$, $\std = \stsn$, $\sjd =\jspn$ for dual ones.

Let 
$\meshsymb_h$ be a triangular virtual surface mesh of $\Gamma$ built as 
in \cite[Section~4.1]{CGH21},
with target element size $h$, and let $\check{\meshsymb}_h$ be its dual as realised on the barycentric refinement \cite{BuC07}. 
It is worth noticing that the BE spaces above can be chosen as 
\begin{itemize}
\item $\besmtp = \mathcal{S}^{1,0}(\meshsymb_h)$: piecewise linear 
``continuous'' functions on $\meshsymb_h$,
\item $\besmtd = \mathcal{S}^{0,-1}(\check{\meshsymb}_h)$: piecewise constant 
functions on $\check{\meshsymb}_h$,
\end{itemize}
Moreover, the duality pairing $\ll \cdot, \cdot \gg$ preserves the 
duality $\besmtd = \left( \besmtp \right)^\prime$.

\subsubsection{A flawed idea}
\label{sssec:fidea}
With these choices and Corollary~\ref{cor:DisInfSupQSpace}, we have that 
\begin{align}
\sup_{\mu_h \in \besstd} \dfrac{\vert \ll x_h, \mu_h\gg \vert}{\Vert \mu_h 
\Vert_{\stsn}} \geq C_{1} \Vert x_h \Vert_{\jspd}, 
\qquad \text{ and} \qquad
\sup_{u_h \in \besjp} \dfrac{\vert \ll \OW_{\kappa}\, x_h,u_h\gg \vert}{\Vert u_h \Vert_{\jspd}}
\geq C_{\OW} \Vert x_h \Vert_{\jspd}, 
\end{align}
for all $x_h \in \besjp$.
Therefore, it only remains to find an operator $\OB_{\kappa}:\stsn \to \jspd$ 
such that
\begin{align}
\sup_{\mu_h \in \besstd} \dfrac{\vert \ll \OB_{\kappa}\, y_h, \mu_h\gg \vert}{\Vert 
\mu_h \Vert_{\stsn}} \geq C_{\OB} \Vert y_h \Vert_{\stsn}, \quad \forall y_h 
\in \besstd.
\end{align}

Furthermore, based on Calder\'on preconditioning for closed surfaces and its 
applicability to simple screens, one could think of setting $\OB_{\kappa}$ to 
be the weakly singular operator $\OV_{\kappa}$. However, it is clear from 
Lemma~\ref{lem:kernVW} that $\OV_{\kappa}$ will not do the job.

\subsubsection{Changing perspective}
\label{sssec:changingperp}

In order to find the right operator $\OB_{\kappa}$, it is useful to first 
understand what we are looking for. Indeed, when pursuing a quotient-space
discretization $\besjp$ of $\jspd$, it makes sense to study the Galerkin
matrices in the multi-trace discrete spaces $\besmtp$ and $\besmtd$. 

Let $\VA_{\OW}$ be the Galerkin matrix of the hypersingular operator $\OW_{
\kappa}$ on $\besmtp$. Then we know from Lemma~\ref{lem:kernVW} and \cite{CGH21}
that $\ker(\VA_{\OW})=\besstp$ and that GMRES can still solve the arising 
linear system as long as the right hand side vector $\Vg$ is consistent, i.e. 
$\Vg \in \text{Range}(\VA_{\OW})$.

Let $\VP_{\OW}$ be the matrix we will use to (left) precondition $\VA_{\OW}$. 
The first condition we need to satisfy is that the system
\begin{equation}
    \VP_{\OW} \VA_{\OW} \Vu = \VP_{\OW} \Vg
\end{equation}
is consistent. If we choose $\VP_{\OW}$ to be invertible, this is automatically 
satisfied. 
Hence, in order to have a suitable operator preconditioner we need
\begin{itemize}
    \item a stable duality pairing for $\besmtp \times \besmtd$; and
    \item $\VB_{\kappa,h}$ invertible (in $\besmtd$),
\end{itemize}
since this will imply that $\VP_{\OW} = \VM^{-1} \VB_{\kappa,h}  \VM^{-T}$ 
is invertible. Here $\VM$ is the Galerkin matrix of the duality pairing
$\besmtp \times \besmtd$.

\subsubsection{Implementation}
\label{ssec:ImpW}

Note that by construction of the inflated screen, which can be understood as a 
virtual closed surface, the space $\mathcal{S}^{1,0}(\meshsymb_h)$ has one 
degree of freedom at the vertices in $\meshsymb_h \cap \partial \Gamma$. Since 
solutions for the hypersingular equation \eqref{eq:varW} live in $\jspd$, we 
know they will be zero on $\partial \Gamma$ \cite{CGH21}.
Considering that solutions for the hypersingular equation \eqref{eq:varW} live in $\jspd$, and because such functions are only determined up to contributions in $\stsd$, degrees of freedom on $\partial \Gamma$ (which by construction are in $\stsd$) can be safely deleted.
Hence, instead of working with 
$\mathcal{S}^{1,0}(\meshsymb_h)$, we consider  $\mathcal{S}_{0}^{1,0}(\meshsymb_h)\subset \mtsd$ : piecewise linear 
``continuous'' functions on the inflated screen $[\Gamma]$ that 
are zero on $\partial \Gamma$.

When dealing with multi-screens of \textbf{type A}, this will have the computational advantage of allowing us to decouple the BE spaces on each side of the triangular virtual surface mesh $\meshsymb_h$, as depicted in Figure~\ref{fig:geo_all}.

Let us illustrate how we implemented these BE spaces on a multi-screen $\Gamma$ 
that consists of three simple screens $\Gamma_i, i=1,2,3$ meeting at a junction:
\begin{enumerate}
    \item We decompose the inflated multi-screen $[\Gamma]$ as
\begin{equation}
    \label{eq:inflated_screen}
    [\Gamma] = \cup_{l=1}^{3} \inter_{l}
\end{equation}
with $\inter_l = \Gamma_{l} \cup \Gamma_{l+1}$. The normal on $\inter_l$ 
is chosen \emph{outward}. Each simple screen $\Gamma_i$ appears once as the \emph{front} and once as 
the \emph{back} of the multi-screen (Figure~\ref{fig:geo_all}). 

\item For $i=1,2,3$, we create the triangular surface mesh $\Gamma_{i,h}$ 
of $\Gamma_i$ with target element size $h$, and such that the meshes 
$\Gamma_{i,h}$ for $i=1,2,3$ match up along the junction.

We remark that the simple screens $\inter_l$ inherit this mesh. In other 
words, we have $\inter_{l,h}= \Gamma_{l,h} \bigcup \Gamma_{l+1,h}, \, l=1,2,3$.

\begin{figure}
    \centering
    \includegraphics[width=0.4\textwidth]{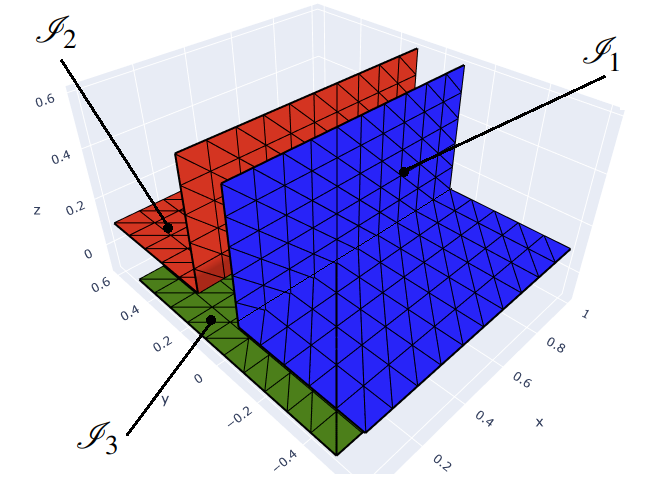}
    \caption{Back-front conforming mesh on the multi-screen.}
    \label{fig:geo_all}
\end{figure}

\item The discrete \emph{primal} multi-trace space is built as the direct product of these spaces, i.e.
\begin{equation}
\besmtp = \mathcal{S}_{0}^{1,0}(\meshsymb_h) = \prod_{l=1}^{3} \mathcal{S}_{0}^{1,0}
(\inter_{l,h}).
\label{eq:primalBES_W}
\end{equation} 

\item Construct the dual BE spaces on the simple screens $\inter_l$ 
following the cue from \cite{BuC07}. For this, let $\check{\inter}_{l,h}$ 
denote the dual barycentric mesh to $\inter_{l,h}$, built as in 
\cite[Definition~2]{HiU16}. Then we introduce the space
$\mathcal{S}^{0,-1}(\check{\inter}_{l,h})\subset \mathbb{H}^{-\half}
(\inter_l)$ of piecewise constant functions supported by the dual cells of $\check{\inter}_{l,h}$ that correspond to nodes not on the boundary of $\inter_l$. In particular we have that $\dim \mathcal{S}^{0,-1}(\check{\inter}_{l,h}) = \dim \mathcal{S}^{1,0}_0(\inter_{l,h})$.

\item The discrete \emph{dual} multi-trace space is built as the direct product of these dual spaces, i.e. 
\begin{equation}
\besmtd = \mathcal{S}^{0,-1}
(\check{\meshsymb}_h)= \prod_{l=1}^{3} \mathcal{S}^{0,-1}(\check{\inter}_{l,h}).
\label{eq:dualBES_W}
\end{equation} 
\end{enumerate}

\begin{remark}
\label{rem:stability}

It is worth noticing that the $L^2(\Gamma)$-duality product between $\besmtp$ 
and $\besmtd$ as chosen in \eqref{eq:primalBES_W} and \eqref{eq:dualBES_W} is 
stable \cite{steinBook2}. Hence, this implementation leads to a Galerkin matrix $\VM$ that is 
bounded and invertible.
\end{remark}

\begin{remark}
The description of discrete multi-trace spaces \eqref{eq:primalBES_W} and \eqref{eq:dualBES_W} is not valid for 
multi-screens of Types B and C. 
\end{remark}

\subsubsection{Block Calder\'on Preconditioner}

Under the considerations of the previous subsections, we propose to use 
Calder\'on preconditioning blockwise. 
This means, 
we will build a preconditioner for the Galerkin matrix for the hypersingular 
operator $\VW_{\kappa,h}$ based on the dual Galerkin matrix 
\begin{align}
    \VB_{\kappa,h}^{\OV} := \left(\begin{array}{c c c}
    \check{\VV}_{\kappa,1} & \mathbf{0} & \mathbf{0} \\
    \mathbf{0} &\check{\VV}_{\kappa,2} &  \mathbf{0}  \\
    \mathbf{0} &\mathbf{0} &\check{\VV}_{\kappa,3}\\
    \end{array}\right),
    \label{eq:BV}
\end{align}
where $\check{\VV}_{\kappa,l}[i,j]=\langle \OV_{\kappa} \check{\psi}_j, 
\check{\psi}_i\rangle_{\inter_{l}}$ with $\check{\psi}_i,\check{\psi}_j$ in the standard basis of $\mathcal{
S}^{0,-1}(\check{\inter}_{l,h})$ for $l=1,2,3$.

The motivation to consider this $\VB_{\kappa,h}^{\OV}$ is that the choice
of discrete spaces from \eqref{eq:primalBES_W} and \eqref{eq:dualBES_W} 
allows us to decouple what is happening on the dual space of each simple 
screen $\inter_{l}$. 
Furthermore, they would agree with the standard discretization of the
 jump spaces on simple screens. More concretely, we have that $\mathcal{S}_{0}^{1,0}
({\inter}_{l,h})\subset \widetilde{H}^{1/2}(\inter_{l})$ and $\mathcal{S}^{0,-1}(\check{\inter}_{l,h}) \subset \widetilde{H}^{-1/2}(\inter_{l})$. 

\begin{proposition}
\label{prop:DisInfSupBV}
Let $\OB_{\kappa,h}^{\OV} \,:\, \besmtd \to \besmtp$ be the linear operator 
corresponding to $\VB_{\kappa,h}^{\OV}$ defined in \eqref{eq:BV}. 
For the discrete spaces defined in this Subsection, 
we have that for all $h \leq h_0$ it holds that 
\begin{align}
    \sup_{u_h\in\besmtd\setminus\{0\}} 
    \dfrac{\vert \ll \OB_{\kappa,h}^{\OV} v_h, u_h \gg \vert}{\Vert u_h \Vert_{\mtsd}}
    \geq \alpha_{\OB_{\OV}} (1 + \vert \log h \vert)^{-2} \Vert v_h \Vert_{\mtsd}
\end{align}
for all $v_h \in \mtsd$, and with $\alpha_{\OB_{\OV}}>0$ independent of $h$. 
\end{proposition}
\begin{proof}
By definition of $\VB_{\kappa,h}^{\OV}$ we have that
\begin{align}
    \sup_{u_h\in\besmtd\setminus\{0\}} 
    \dfrac{\vert \ll \OB_{\kappa,h}^{\OV} v_h, u_h \gg \vert}{\Vert u_h \Vert_{\mtsd}}
    = \sup_{u_h\in\besmtd\setminus\{0\}} 
    \dfrac{\vert \sum_{l} \langle \OW_{\kappa} v_h, u_h \rangle_{\inter_{l}}\vert}{\Vert u_h \Vert_{\mtsd}}.   
    \label{eq:DISBV1}
\end{align} 

Recall that $\OW_{\kappa}$ satisfies a G\aa{}rding inequality on each $\inter_l$ with a compact operator $\OT_{\OW,l}$. Let $\OD_{\OW, l} := \OW_{\kappa}
+ \OT_{\OW,l}$ and denote by $u_{hl}$ the restriction
of $u_h$ to $\inter_{l}$. Then, we choose $u_h$ such that $u_{hl} = v_{hl} - 
{\OD_{\OW, l}}^{-1} \OT_{\OW,l}v_{hl}$. Plugging this into \eqref{eq:DISBV1} gives
\begin{align}
    \sup_{u_h\in\besmtd\setminus\{0\}} 
    \dfrac{\vert \ll \OB_{\kappa,h}^{\OV} v_h, u_h \gg \vert}{\Vert u_h \Vert_{\mtsd}}
    \geq \dfrac{\vert \sum_{l} \langle \OW_{\kappa} v_{hl}, v_{hl} - 
{\OD_{\OW, l}}^{-1} \OT_{\OW,l}v_{hl} \rangle_{\inter_{l}}\vert}{\Vert u_h \Vert_{\mtsd}}
    \label{eq:DISBV2}
\end{align} 
Next, following standard arguments (\emph{c.f}\cite[Theorem~8.11]{steinBook}), 
one gets that for this choices of $u_{hl}$ there exists an $\tilde{h}_{l}\in \mathbb{R}_+$ such that 
\begin{align*}
 \langle \OW_{\kappa} v_{hl}, u_{hl} \rangle_{\inter_{l}} \geq 
\tilde{c}_l \Vert v_{hl} \Vert_{\widetilde{H}^{1/2}(\inter_{l})} 
\Vert u_{hl}  \Vert_{\widetilde{H}^{1/2}(\inter_{l})} 
\geq 
\tilde{c}_l \Vert v_{hl} \Vert_{{H}^{1/2}(\inter_{l})} 
\Vert u_{hl}  \Vert_{{H}^{1/2}(\inter_{l})} 
\end{align*}
is satisfied for all $h\leq \tilde{h}_l$, and with $\tilde{c}_l>0$ independent of $h$.
Hence, we get
\begin{align}
    \sup_{u_h\in\besmtd\setminus\{0\}} 
    \dfrac{\vert \ll \OB_{\kappa,h}^{\OV} v_h, u_h \gg \vert}{\Vert u_h \Vert_{\mtsd}}
    \geq \dfrac{\sum_{l} \tilde{c}_l \Vert v_{hl} \Vert_{{H}^{1/2}(\inter_{l})} 
\Vert u_{hl}  \Vert_{{H}^{1/2}(\inter_{l})} }{\Vert u_h \Vert_{\mtsd}},
    \label{eq:DISBV3}
\end{align} 
for all $h\leq h_0:= \min_{l} \tilde{h}_l$.

Now, let $c_*= \min \tilde{c}_l$ and use Polya and Szeg\"o's inequality to 
further bound our expression as follows
\begin{align}
    \sup_{u_h\in\besmtd\setminus\{0\}} 
    \dfrac{\vert \ll \OB_{\kappa,h}^{\OV} v_h, u_h \gg \vert}{\Vert u_h \Vert_{\mtsd}}
    \geq c_* \dfrac{\sum_{l_1}\Vert v_{hl_1} \Vert_{{H}^{1/2}(\inter_{l_1})} 
\sum_{l_2}\Vert u_{hl_2}  \Vert_{{H}^{1/2}(\inter_{l_2})} }{\Vert u_h \Vert_{\mtsd}}.
    \label{eq:DISBV4}
\end{align} 
Finally, using the inverse inequality from Lemma~\ref{lem:invIneqDir}, we conclude 
\begin{align}
    \sup_{u_h\in\besmtd\setminus\{0\}} 
    \dfrac{\vert \ll \OB_{\kappa,h}^{\OV} v_h, u_h \gg \vert}{\Vert u_h \Vert_{\mtsd}}
    \geq \tilde{c}_* (1+ \vert \log h \vert)^{-2} \Vert v_h \Vert_{\mtsd}.
    \label{eq:DISBV5}
\end{align} 

\end{proof}

\subsubsection{Condition number estimates}

Although GMRES convergence estimates do not rely only on spectral 
condition numbers, it is often a useful piece of information in the 
context of operator preconditioning because it gives us a simple criteria 
to preserve stability, study asymptotic behaviours and to compare our 
preconditioning results with what is known in the literature for simple screens. Moreover, 
as we will see later, it will help us provide criteria to choose 
smaller discrete spaces that are still amenable to efficient preconditioning.

\begin{theorem}
\label{thm:condCaldPrecW}
Let $\VW_{\kappa,h}$ be the Galerkin matrix corresponding to $\OW_{\kappa}$
discretized over $\besmtp \subset \mtsd$, and $\VM$ the Galerkin matrix of the 
duality pairing for $\besmtp\times \besmtd$ as chosen above.

Assume that there exists an operator $\OR_h^+ \, : \, \mtsd \to \besmtp$ such that
\begin{itemize}
 \item $\OR_h^+$ is a h-uniformly bounded projection
 \item $\OR_h^+ ( \stsd ) \subseteq  \besstp. $
\end{itemize}

Then, under the mesh conditions from Assumption~\ref{asm:MeshAssumption}, we have
\begin{align}
    \kappa_{sp}(\VM^{-1} \VB_{\kappa,h}^{\OV}\VM^{-T} \VW_{\kappa,h}) 
    \leq (1 + \vert \log h \vert)^2\dfrac{\alpha_{\OM}^2
\Vert \OW_{\kappa} \Vert \Vert \OV_{\kappa} \Vert}{\Vert \OM \Vert^2 \alpha_{\OB} \alpha_{\ba}},
\end{align}
where $\Vert \Vert$ denotes operator norms and $\alpha_(\cdot)$ 
corresponding inf-sup constants.
\end{theorem}
\begin{proof}
Given the inf-sup constants from Corollary~\ref{cor:DisInfSupQSpace} and 
Proposition~\ref{prop:DisInfSupBV}, and the norm equivalences 
from Lemmas \ref{lem:normEquivalenceD} and \ref{lem:normEquivalenceN} shown in 
Appendix~\ref{app:NormEquiv}, the result follows from the derivation in
Appendix~\ref{app:CondNumEstimates}.
\end{proof}

\subsection{Preconditioning the Weakly singular operator}
\label{ssec:precV}

Now we are interested in the Dirichlet variational problem \eqref{eq:varV}, where we have $\smtp=\mtsn$, $\stp = \stsn$, and $\sjp = \jspn$ for the primal spaces; and 
$\smtd=\mtsd$, $\std = \stsd$, and $\sjd =\jspd$ for the dual ones.

This time, we chose these BE spaces as 
\begin{itemize}
\item $\besmtp = \mathcal{S}^{0,-1}(\meshsymb_h)$: piecewise constant 
functions on $\meshsymb_h$,
\item $\besmtd = \mathcal{S}^{1,0}(\check{\meshsymb}_h)$: the piecewise linear, 
continuous functions on $\check{\meshsymb}_h$ as built in \cite{BuC07},
\end{itemize}
Moreover, the duality pairing $\ll \cdot, \cdot \gg$ preserves the 
duality $    \besmtd = \left( \besmtp \right)^\prime$.

In analogy to what we discussed in subsection~\ref{ssec:precW}, we have that 
standard Calder\'on preconditioning, i.e. using $\OW_{\kappa}$ to precondition 
$\OV_{\kappa}$ will not work. Hence, we will again consider a block diagonal
Calder\'on preconditioner:
\begin{align}
    \VB_{\kappa,h}^{\OW} := \left(\begin{array}{c c c}
    \check{\VW}_{\kappa,1} & \mathbf{0} & \mathbf{0} \\
    \mathbf{0} &\check{\VW}_{\kappa,2} &  \mathbf{0}  \\
    \mathbf{0} &\mathbf{0} &\check{\VW}_{\kappa,3}\\
    \end{array}\right),
    \label{eq:BW}
\end{align}
where $\check{\VW}_{\kappa,l}[i,j]=\langle \OW_{\kappa} \check{\psi}_j, 
\check{\psi}_i\rangle_{\inter_{l}}$ with $\check{\varphi}_i,\check{\varphi}_j$ in the standard basis for $\mathcal{
S}^{1,0}(\check{\inter}_{l,h})$ for $l=1,2,3$. Then, we can show
\begin{proposition}
\label{prop:DisInfSupBW}
Let $\OB_{\kappa,h}^{\OW} \,:\, \besmtd \to \besmtp$ be the linear operator 
corresponding to $\VB_{\kappa,h}^{\OW}$ defined in \eqref{eq:BW}. For the discrete spaces defined in this Subsection, 
we have that for all $h \leq h_0$ it holds that 
\begin{align}
    \sup_{u_h\in\besmtd\setminus\{0\}} 
    \dfrac{\vert \ll \OB_{\kappa,h}^{\OW} \varphi_h, \mu_h \gg \vert}{\Vert \mu_h \Vert_{\mtsn}}
    \geq \alpha_{\OB_{\OW}} (1 + \vert \log h \vert)^{-2} \Vert \varphi_h \Vert_{\mtsn}
\end{align}
for all $\varphi_h \in \mtsn$, and with $\alpha_{\OB_{\OW}}>0$ independent of $h$. 
\end{proposition}
\begin{proof}
The proof follows is analogous to the proof of Proposition~\ref{prop:DisInfSupBV}, but 
using Lemma~\ref{lem:invIneqNeu}.
\end{proof}

Finally, following the same steps as in Theorem~\ref{thm:condCaldPrecW}, we 
arrive to the following condition number estimate:
\begin{theorem}
\label{thm:condCaldPrecV}
Let $\VV_{\kappa,h}$ be the Galerkin matrix corresponding to $\OV_{\kappa}$
discretized over $\besmtp \subset \mtsn$, and $\VM$ the Galerkin matrix of the 
duality pairing for $\besmtp\times \besmtd= \mathcal{S}^{0,-1}(\meshsymb_h)\times
\mathcal{S}^{1,0}(\check{\meshsymb}_h)$.

Assume that there exists an operator $\OR_h^- \, : \, \mtsn \to \besmtp$ such that
\begin{itemize}
 \item $\OR_h^-$ is a h-uniformly bounded projection
 \item $\OR_h^- ( \stsn ) \subseteq  \besstp. $
\end{itemize}

Then, under mesh conditions from Assumption~\ref{asm:MeshAssumption}, we have
\begin{align}
    \label{eq:boundonkappa}
    \kappa_{sp}(\VM^{-1} \VB_{\kappa,h}^{\OW}\VM^{-T} \VV_{\kappa, h}) 
    \leq (1 + \vert \log h \vert)^2\dfrac{\alpha_{\OM}^2
\Vert \OW_{\kappa} \Vert \Vert \OV_{\kappa} \Vert}{\Vert \OM \Vert^2 \alpha_{\OB_{\OW}} \alpha_{\ba_{\OV}}}.
\end{align}
\end{theorem}

\begin{remark}
    It is worth pointing out that, although we do not discuss the existence of projection operators $\OR_h^+$ and $\OR_h^-$ in this article, this is a reasonable assumption for us to make. Indeed, the operator $\OR_h^+$ was built in \cite{Ave22} and a similar approach may be possible to construct $\OR_h^-$.
\end{remark}

\section{Calder\'on preconditioning on Reduced Quotient-Space BEM}
\label{sec:CaldPreRed}

The above analysis has been carried out for the case where the finite element space is chosen to approximate the entire multi-trace space. Since the solution is determined in the jump space, it can be worth while to investigate whether (combinations of) degrees of freedom (DoFs) can be deleted and whether the resulting method remains amenable to operator preconditioning schemes.

In this section we will introduce several ways in which the number of DoFs can be reduced, what mileage can be expected from the resulting methods, and we discuss what the ramifications are for implementations in code of these methods.

The most straightforward approach to building a well-conditioned boundary element method on multi-screens is to introduce a finite element space for the multi-trace space that is contained in the direct product space. Two key ingredients for the success of this approach are that
\begin{itemize}
\item the discrete left/right nullspace $\besstp$ equals $\stsd \cap \besmtp$; and that
\item the quotient $\besmtp / \besstp$ approximates $\jspd$.
\end{itemize}
  However, because we are interested in finding an approximate solution in the quotient space $\jspd$, we are free to consider boundary element spaces $\besmtp^\circ$ that do not approximate all of $\mtsd$ as long as the corresponding discrete nullspace $\besstp^\circ=\stsd \cap \besmtp^\circ$ is still a subset of $\stsd \cap \besmtp$ and the quotient spaces $\besmtp^\circ / \besstp^\circ$ and $\besmtp / \besstp$ are equal. Similar choices can be made to select a reduced dual finite element space $\besmtd^\circ \subseteq \besmtd$. The quality of the resulting operator preconditioning depends on the stability of the restriction of the duality form to this subspace.

How does this work in practice? In the case of nodal elements in $\mathcal{S}_{0}^{1,0}(\meshsymb_h)$, any given basis function relates to a function in $\besstp$ by completing it with its counterpart(s) on the opposite side(s) of the multi-screen. By removing one of the basis functions from the standard nodal basis for $\besmtp^\circ$, the dimension of the discrete nullspace $\besstp^\circ$ goes down by one. The dimension of the complement of $\besstp^\circ$ remains unchanged and so necessarily $\besmtp^\circ / \besstp^\circ = \besmtp / \besstp$. To put it in more physical terminology: the reduced discrete multi-trace space $\besmtp^\circ$ \emph{radiates} the same fields as the original one.

There are a number of reduction strategies that are fairly straightforward to implement. We will discuss here three strategies that can be applied to a multi-screen $\Gamma$ comprising a single junction where an odd number $m$ simple screens meet.
\begin{itemize}

    \item[(i)] \emph{Partial reduction}: In the partition $[\Gamma] = \cup_{i=1}^{m} \inter_i$, degrees of freedom based on the terms $i=3,5,7,...$ can be discarded. This is extremely easy to implement and boils down to using $\besmtp^\circ = \mathcal{S}^{0,-1}(\inter_{1,h}) \times \prod_{i=1}^{\lfloor m/2 \rfloor} \mathcal{S}^{0,-1}(\inter_{2i,h})$ instead of $\besmtp = \prod_{i=1}^m \mathcal{S}^{0,-1}(\inter_{i,h})$.
    
        \item[(ii)] \emph{Single strip}: The partial reduction described above in essence removes the \emph{back} from part of $[\Gamma]$. This still leaves significant redundancy in $\besmtp^\circ$. In our example leaving out $\mathcal{S}^{0,-1}(\inter_{i,h})$ for $i=3,5,7,...$ still leaves all the DoFs on $\Gamma_2$ (excluding DoFs on the junction) that can be completed by DoFs on \emph{the other side} to yield functions in $\besstp$. As a result, we can further discard DoFs in $\mathcal{S}^{0,-1}(\inter_{1,h})$ that lie in the interior of $\Gamma_2$. If the implementer has access to node-triangle adjacency information this approach requires minimal coding effort. The resulting finite element space is $\besstp^\circ = \mathcal{S}^{0,-1}(\inter_{1,h})^\circ \times \prod_{i=1}^{\lfloor m/2 \rfloor} \mathcal{S}^{0,-1}(\inter_{2i,h})$, where the $\circ$ superscript on the first factor denotes that this finite element space is reduced by leaving out redundant DoFs linked to nodes that are in $\inter_{1,h}\cap\Gamma_2$ but not on the junction.
    
    \item[(iii)] \emph{Fixed overlap}: Note that the efficiency of the resulting preconditioning method depends on the lower bound for the duality form $\ll.,.\gg$ on $\besmtp^\circ \times \besmtd^\circ$, which may depend on the geometry and hence may indirectly depend on $h$ when using a single strip reduction. In those situations where this is undesirable, one can opt to leave in not only those DoFs in $\mathcal{S}^{0,-1}(\inter_{1,h})$ that are positioned on $\Gamma_1$ or on the junction, but also those on $\Gamma_2$ inside a strip within a fixed mesh independent distance from the junction. Likely, this will require manipulations to the code at the level of mesh generation. The resulting method will lead to an increasing redundancy in DoFs as $h$ tends to zero, but the user is guaranteed that the preconditioner efficiency will not be limited by degradation of the duality pairing stability.
\end{itemize}

We illustrate these three reduction strategies for $m=3$ on Figure~\ref{fig:reductions}.
We also point out that when $m$ is even, all these reductions are also valid, but since one can always find a \emph{partial reduction} that provides a minimal representation of the quotient space, i.e. $\besmtp^\circ = \prod_{i=1}^{m/2} \mathcal{S}^{0,-1}(\inter_{2i,h})$, the other two proposed strategies are not computationally attractive.

Finally, it is worth mentioning that regardless the choice of reduction method and the corresponding primal finite element space $\besmtp^\circ$, the construction of the dual finite element space $\besmtd^\circ$ remains the same. The construction goes along the lines of what is described in \cite{BuC07}, starting from the reduced surfaces and corresponding meshes
\begin{equation}
    \inter_i \cap \bigcup_{u \in \besmtp^\circ} \operatorname{supp} u
\end{equation}
This means in particular that for the Dirichlet problem, the dual space of piecewise linear, continuous elements is attains non-zero values on the boundary of the supporting reduced mesh, as detailed in \cite{BuC07}. This may seem counterintuitive but is required for the discrete stability of the duality form.

Moreover, since the discrete stability of the duality form also implies the continuity estimates \eqref{eq:contQjh} and \eqref{eq:contPih} used in our proofs. Therefore, we have that by ensuring this stability, all results in the appendix can be extended to the proposed reduced Quotient-space BEM and hence we are still within the framework of Theorems~\ref{thm:condCaldPrecW} and \ref{thm:condCaldPrecV}. For this, it is crucial to identify the reduced primal space on $\inter_{1,h}$ with the (complete) space on a truncated simple screen 
$\inter_{1,h}^\circ$, which are displayed in blue in Figures~\ref{fig:fRed} and \ref{fig:sRed}. So for example, one identifies $\mathcal{S}^{0,-1}(\inter_{1,h})^\circ$ with $\mathcal{S}^{0,-1}(\inter_{1,h}^\circ)$.
We refer the reader to Appendix~\ref{app:CondNumEstimates} for further details.

\begin{figure}
     \centering
     \begin{subfigure}[b]{0.4\textwidth}
         \centering
         \includegraphics[width=\textwidth]{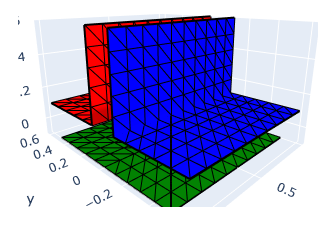}
         \caption{Full multi-trace discretisation.}
         \label{fig:noRed}
     \end{subfigure}
     \hfill
     \begin{subfigure}[b]{0.4\textwidth}
         \centering
         \includegraphics[width=\textwidth]{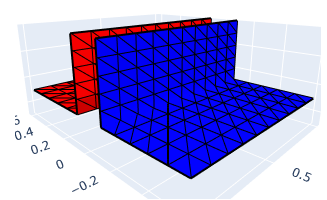}
         \caption{Partial reduction.}
         \label{fig:pRed}
     \end{subfigure}
     \hfill
     \begin{subfigure}[b]{0.4\textwidth}
         \centering
         \includegraphics[width=\textwidth]{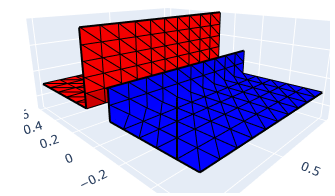}
         \caption{Fixed overlap}
         \label{fig:fRed}
     \end{subfigure}
     \hfill
     \begin{subfigure}[b]{0.4\textwidth}
         \centering
         \includegraphics[width=\textwidth]{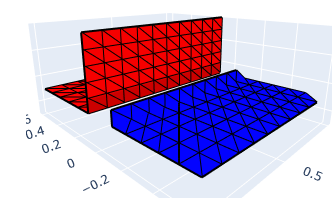}
         \caption{Single strip}
         \label{fig:sRed}
    \end{subfigure}
    \caption{Meshes illustrating full multi-trace discretization and three reduction strategies used here.}
    \label{fig:reductions}
\end{figure}

\section{Numerical Results}

\subsection{Preconditioning the hypersingular operator (Neumann problem)}
Consider the geometry in Figure~\ref{fig:reductions}. The structure is illuminated by a plane wave with signature $
    \exp \left(i\kappa x_3\right)$.
The numerical experiments will be run in the low-frequency regime ($\kappa = 1$) and the moderate frequency regime ($\kappa = 10$).

Linear systems are solved using GMRES with the tolerance set to $2.0e-5$. To build the preconditioners, application of the inverse Gram matrix is required. This action is computed by running a second, inner GMRES solver within the outer, primal solver. Numerical experiments have shown that it is important to set the tolerance for this inner GMRES sufficiently low. In the experiments presented here the tolerance is set to $2.0e-12$. Fortunately the Gram matrices are well conditioned and application of their inverses through GMRES can be computed in a small and linear number of operations, even at these very small tolerances.

Another important aspect of implementing the preconditioning strategies presented above is the use of high quality quadrature rules, especially for interactions between geometric elements that are close together. Specifically, it is important that left and right nullspaces of the discrete bilinear forms are invariant upon introduction of the quadrature error. One can either choose to adopt highly accurate quadrature rules or to use rules that are symmetric with respect to back-front mirroring across the multi-screen. Here we have opted for the highly accurate and kernel independent Sauter-Schwab rules described in Chapter 5 of \cite{SaS11}.

\begin{figure}[h]
    \centering
    \includegraphics[width=0.6\linewidth]{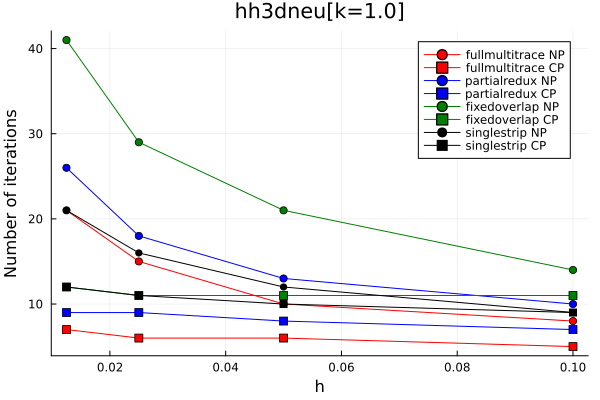}
    \caption{GMRES iterations vs $h$ at $\kappa = 1$ for the Neumann problem at $\Gamma$ as in Figure~\ref{fig:reductions}.}
    \label{fig:NeumannLF}
\end{figure}

\begin{figure}[h]
    \centering
    \includegraphics[width=0.6\linewidth]{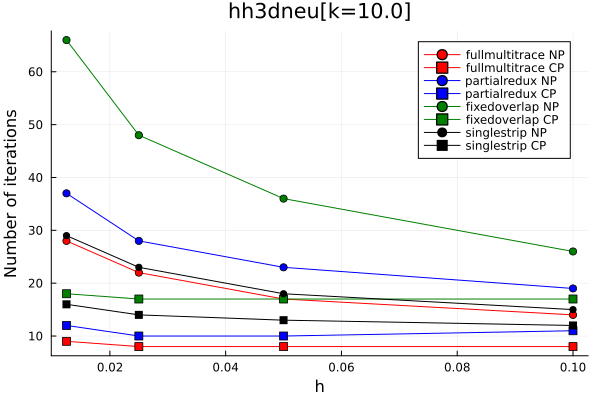}
    \caption{GMRES iterations vs $h$ at $\kappa = 10$ for the Neumann problem at $\Gamma$ as in Figure~\ref{fig:reductions}.}
    \label{fig:NeumannHF}
\end{figure}

The obtained results are displayed in Figure~\ref{fig:NeumannLF}
for $\kappa = 1$ and in Figure~\ref{fig:NeumannHF} for $\kappa = 10$. There we label iteration counts for the unpreconditioned system by NP, and those for after Calder\'on preconditioning by CP. In both the low frequency and the moderate frequency case we find a much smaller number of iterations is required upon application of our operator preconditioning approach. For all four considered reductions of the $\besmtp$ depicted in Figure \ref{fig:reductions}, there is a clear improvement. After preconditioning there remains a slow increase in the number of iterations, commensurate with the logarithmic grow in the \eqref{eq:boundonkappa}.

At moderate frequencies, both the original system and the preconditioned system require more iterations, but the benefits of applying operator preconditioning remain.

\subsection{Preconditioning the weakly singular operator (Dirichlet problem)}

We consider the same geometry, excitation and GMRES tolerance used to study our preconditioner for the Neumann problem. 
An important difference is that DoFs for the Dirichlet problem are linked to triangles of the mesh, as opposed to vertices. The support of the primal basis functions spans only a single triangle. The reduction of the multi-trace space can be done up to the point where there is no overlap between the simple screens that support the reduced finite element spaces.

\begin{figure}[h]
    \centering
    \includegraphics[width=0.6\linewidth]{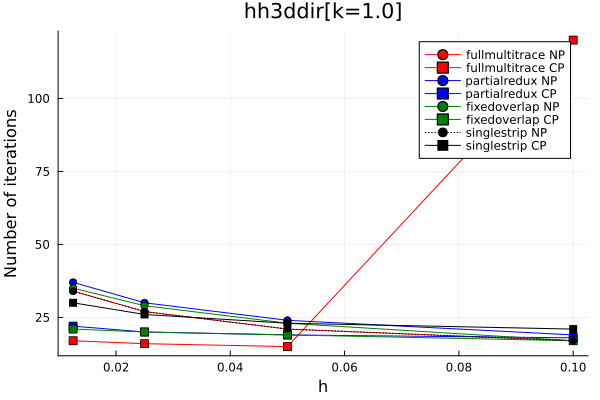}
    \caption{GMRES iterations vs $h$ at $\kappa = 1$ for the Dirichlet problem at $\Gamma$ as in Figure~\ref{fig:reductions}.}
    \label{fig:DirichletLF}
\end{figure}

\begin{figure}[h]
    \centering
    \includegraphics[width=0.6\linewidth]{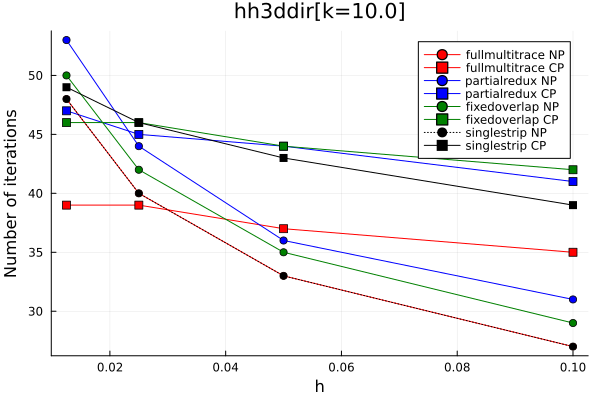}
    \caption{GMRES iterations vs $h$ at $\kappa = 10$ for the Dirichlet problem at $\Gamma$ as in Figure~\ref{fig:reductions}.}
    \label{fig:DirichletHF}
\end{figure}

Essentially all conclusions drawn for the Neumann problem carry over to the study of the numerical solution of the Dirichlet problem. In Figure~\ref{fig:DirichletLF} and Figure~\ref{fig:DirichletHF} it can be seen that at both frequencies and for all reduction strategies there is a clear decrease in the number of iterations required for solution. At moderate frequencies the performance of our preconditioner is less outspoken than for the Neumann problem. In fact, for the specific choice of the most aggressive reduction scheme the unpreconditioned system requires fewer iterations than the preconditioned system for all values of the mesh size $h$ we have investigated. Nevertheless, the trend in the corresponding lines in Figure~\ref{fig:DirichletHF} is such that a cross-over point is to be expected at only modestly smaller values for $h$.

\subsection{Application to multi-screens of Type B}

\begin{figure}
    \centering
    \includegraphics[width=0.4\textwidth]{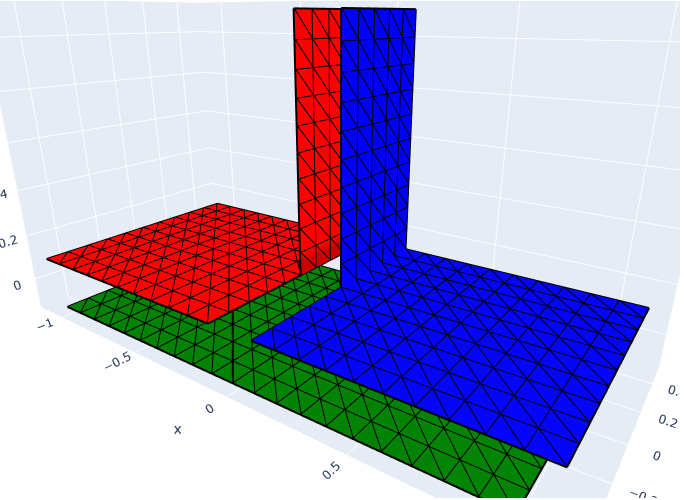}
    \includegraphics[width=0.4\textwidth]{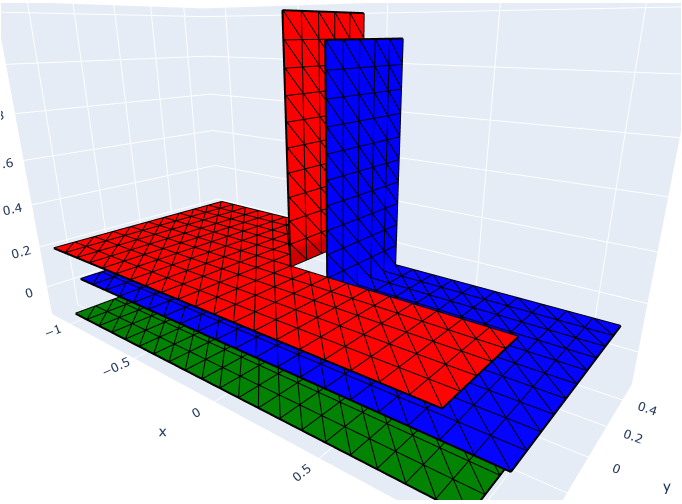}
    \caption{Two possible coverings for $[\Gamma]$. The most economic covering on the left precludes the definition of finite element spaces of direct product type (left). Allowing part of $[\Gamma]$ to be multiply covered resolves this problem (right).}
    \label{fig:geo_special}
\end{figure}

For multi-screens of Type B, slight modifications to the choice of finite element spaces are required in order to arrive at a linear system requiring only few iterations for its solution. It may seem most natural to write $[\Gamma]$ as the union of the simple screens depicted on the left in Figure~\ref{fig:geo_special}. Unfortunately, this partitioning does not allow the construction of a finite element space $\besmtp$ that can be written as the direct product of finite element spaces supported by the $\inter_i$. The issue is that the solution for the Neumann problem in general will not be in $\prod_i \widetilde{H}^{1/2}(\inter_i)$ and that as a result degrees of freedom along the segment from $(0,0,0)$ to $(0,-0.5,0)$ cannot be discarded.

Allowing overlapping coverings of $[\Gamma]$ as depicted on the right of Figure~\ref{fig:geo_special} resolves this problem. For $l=1,..,m$, let $\bar{\inter}_{l}$ denote either $\inter_l$ with overlap when needed, or without overlap. We can use the finite element space $\prod_{i=1}^m \mathcal{S}^{1,0}(\bar{\inter}_{i,h})$, which in a sense is larger than what we need but still leads to the correct quotient space.

We use this approach to solve the Neumann problem for the geometry in Figure~\ref{fig:geo_special} and for excitation $\exp(-i\kappa x_3)$ with $\kappa=10$. Figure~\ref{fig:hh3dneu_augmented} demonstrates that upon preconditioning the number of iterations is much lower than what is required to solve the original linear system when solving the Neumann problem.

Figure~\ref{fig:hh3ddir_augmented} shows thee results for the Dirichlet problem. They are in line with those from Figigure~\ref{fig:DirichletHF}: the higher offset in the iteration count results in a cross-over point at smaller values of $h$, but asymptotically the preconditioner leads to a more efficient algorithm.

The numerical results presented in this section have been produced with the boundary element package BEAST.jl\footnote{\texttt{https://github.com/krcools/BEAST.jl}}. The scripts to reproduce them cam be found in a public Github repository\footnote{\texttt{https://github.com/krcools/Junctions\_KC\_CUT.jl}}.

\begin{figure}[h]
    \centering
    \includegraphics[width=0.6\textwidth]{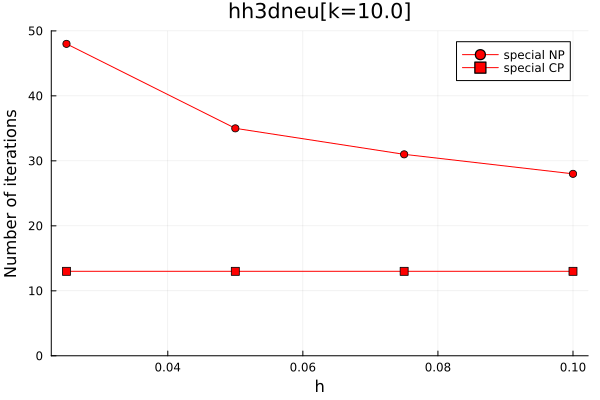}
    \caption{GMRES iterations vs $h$ at $\kappa = 10$ for the Neumann problem
    for scattering by a geometry of type B.}
    \label{fig:hh3dneu_augmented}
\end{figure}

\begin{figure}[h]
    \centering
    \includegraphics[width=0.6\textwidth]{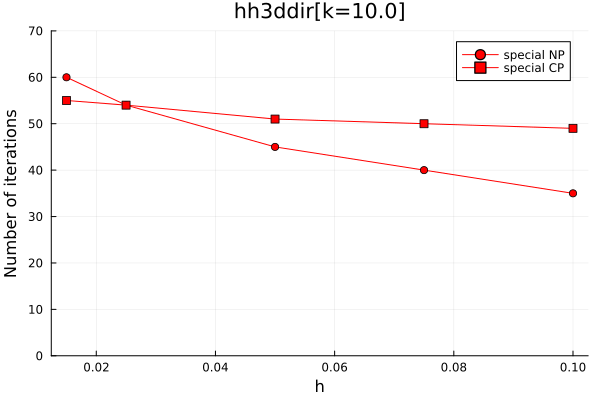}
    \caption{GMRES iterations vs $h$ at $\kappa = 10$ for the Dirichlet problem
    for scattering by a geometry of type B.}
    \label{fig:hh3ddir_augmented}
\end{figure}

\section{Conclusions}

We have presented an effective Calder\'on-type preconditioner 
for Helmholtz equations at multi-screens that builds on quotient-space 
BEM and operator preconditioning. Moreover, we have 
proved and confirmed numerically that it performs as 
standard Calder\'on preconditioning does on simple screens. 

From a computational point of view, quotient-space
BEM considering the full discretization of multi-valued traces 
has the advantage of requiring minimal geometrical information 
but the disadvantage of "doubling DoFs". As an alternative, we 
proposed different strategies to work with reduced multi-trace 
discretizations that use less DoFs but require more adaptations when using a standard BEM code. We gave details regarding 
the additional data requirements in the implementation of all 
these strategies, and used the developed 
framework to identify the requirements that reduced spaces need 
to meet in order to still have efficient Calder\'on-type preconditioning.

Finally, we briefly presented an heuristic strategy to precondition multi-screens that also appear in applications but that are not covered by our 
theory. Although in essence our approach follows the same principles of our
Calder\'on-type preconditioner for type A multi-screens, rigorous analysis has been elusive and therefore has not been treated in this article. Indeed, the key missing piece is an extension of Lemma~8 for this case. Nevertheless, we offered numerical experiments to investigate its effectivity.

Current and future work also involves extending the analysis of this preconditioning approach to Maxwell equations, where numerical results 
are promising \cite{CUT22}.

\subsubsection*{Funding}
This project has received funding from the European Research Council (ERC) under 
the European Union’s Horizon 2020 research and innovation programme 
(Grant agreement No. 101001847) and from  the Dutch Research Council (NWO) under the NWO-Talent Programme Veni with the project number VI.Veni.212.253.

\bibliographystyle{plain}
\bibliography{biblio}

\appendix
\section{Auxiliary Lemmas}

In this Section we prove some auxiliary results that hold for the multi-screens 
under consideration. We remark that for this we follow the cue from 
\cite[Section~5.2]{ClH13} and use the properties of the associated volume-based 
spaces.

Let us begin by noticing that the Lipschitz partition $\left( \Omega_j 
\right)_{j=0\ldots n}$ such that $\Gamma \subset \cup_{j=0}^m \partial 
\Omega_j$ is not unique. We illustrate this for a two-dimensional triple junction $\Gamma$. in 
Fig.~\ref{fig:LP}. 
Nevertheless, for the multi-screens considered in this paper, one can always 
find a Lipschitz partition such that $\Gamma\cap\partial\Omega_i\cap\partial 
\Omega_j\neq\emptyset$ for all $i,j=0,\dots,m$. In other words, we can always 
assume we have the configuration corresponding to Fig.~\ref{fig:LPMin}.
For simplicity of the proofs, this is the type of Lipschitz partitions that we 
will consider. This and the particular order of the domains is stated in the 
following:

\begin{figure}
\centering
 \begin{subfigure}[b]{0.45\textwidth}
     \includegraphics[width=\textwidth]{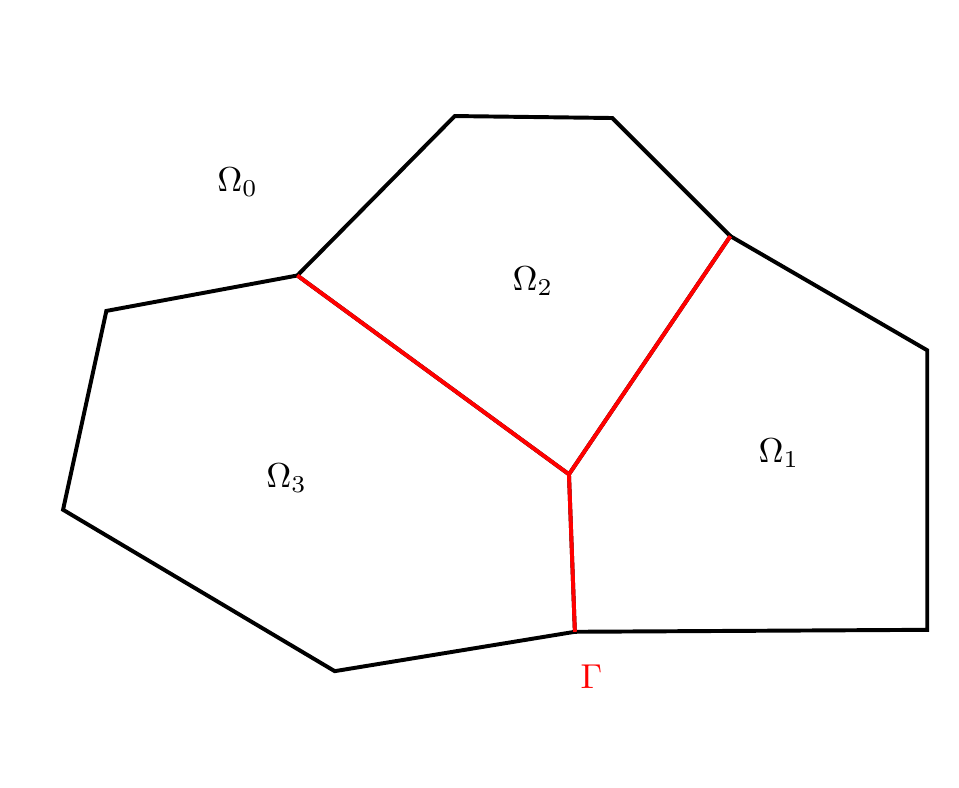}
     \caption{ }
     \label{fig:LPArb}
 \end{subfigure}
 \hfill
 \begin{subfigure}[b]{0.45\textwidth}
     \includegraphics[width=\textwidth]{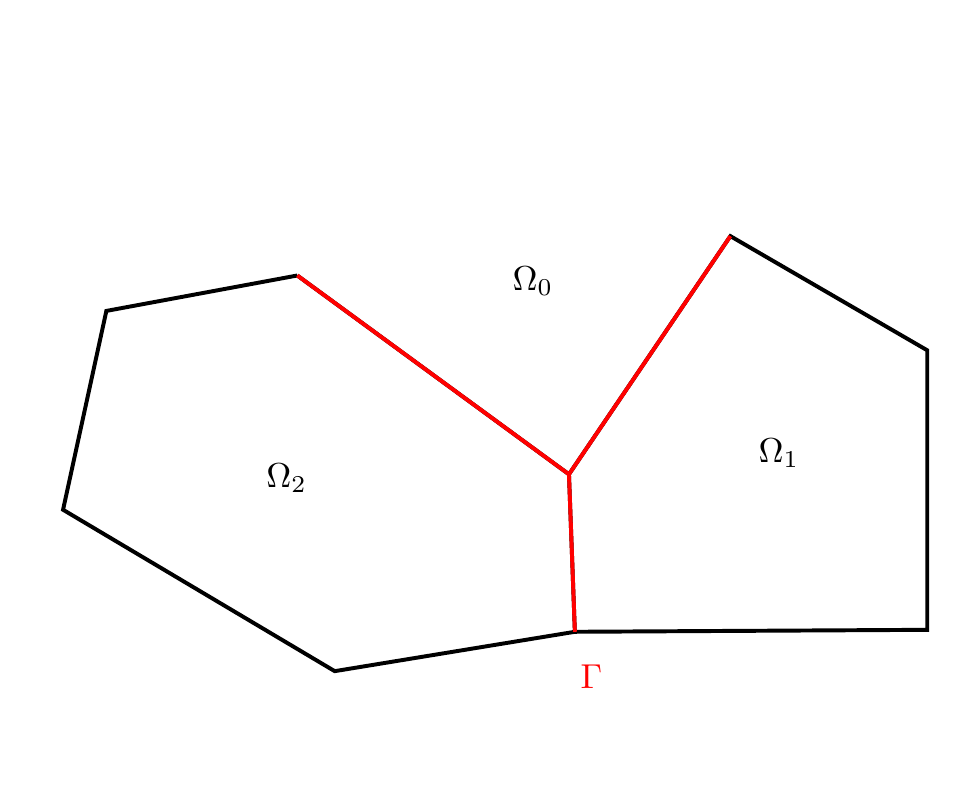}
     \caption{ }
     \label{fig:LPMin}
 \end{subfigure}
 \caption{Example of two Lipschitz partitions for $\Gamma$ being a multi-screen 
 with a triple junction.}
 \label{fig:LP}
\end{figure}

\begin{assumption}
\label{ass:ms}
Let $\left( \Omega_j \right)_{j=0\ldots n}$ be a Lipschitz partition in $\mathbb{R 
}^3$. We assume $\Gamma$ to be a multi-screen such that $\Gamma\subset\cup_{j=0 
}^m\partial \Omega_j$ and $\Gamma\cap\partial\Omega_i\cap\partial\Omega_j \neq 
\emptyset$ for all $i,j=0,\dots,m$.
Moreover, and without loss of generality, we assume that $\Gamma$ and the 
Lipschitz partition $\left( \Omega_j \right)_{j=0\ldots n}$ are such that 
$\Omega_0$ is the exterior domain. 
\end{assumption}

In order to improve readability of our auxiliary lemmas, let us define $S_i :=
\Gamma \cap \partial \Omega_i$ for $i=0,\dots,m$.
Now we are in the position to introduce the first result of this Appendix.
\begin{lemma}
\label{lem:MTinjections}
Let $\Gamma$ be a multi-screen as in Assumption~\ref{ass:ms}. Then, the 
following injections hold
\begin{align*}
\mtsd \xhookrightarrow{} H^{1/2}(S_0)\times \dots \times H^{1/2}(S_m), \\
\mtsn \xhookrightarrow{} H^{-1/2}(S_0)\times \dots \times H^{-1/2}(S_m).
\end{align*}
\end{lemma}
\begin{proof}
We recall that $\mtsd = H^1(\mathbb{R}^d\backslash \overline{\Gamma})/ H^1_{0,
\Gamma}(\mathbb{R}^d)$ and note that $H^1(\mathbb{R}^d\backslash 
\overline{\Gamma}) \subset H^1(\mathbb{R}^d\backslash \overline{\cup_{j=0}^m 
\partial \Omega_j})$. This induces the injection 
\begin{align*}
    \mtsd = H^1(\mathbb{R}^d\backslash \overline{\Gamma})/ H^1_{0,\Gamma}( 
    \mathbb{R}^d) \xhookrightarrow{} H^1(\mathbb{R}^d\backslash 
    \overline{\cup_{j=0}^m \partial \Omega_j})/H^1_{0,\Gamma}(\mathbb{R}^d).
\end{align*}
Additionally, we have the natural identification
\begin{align*}
    H^1(\mathbb{R}^d\backslash \overline{\cup_{j=0}^m \partial \Omega_j})
    \cong H^1(\Omega_0) \times \dots H^1(\Omega_m),
\end{align*}
that associates $u \in H^1(\mathbb{R}^d\backslash \overline{\cup_{j=0}^m 
\partial\Omega_j})$ with $(u_{|\Omega_0}, \dots, u_{|\Omega_m})$. From this 
natural identification, we get the isomorphism
\begin{align*}
H^1(\mathbb{R}^d\backslash 
    \overline{\cup_{j=0}^m \partial \Omega_j})/H^1_{0,\Gamma}(\mathbb{R}^d)
    &\cong [H^1(\Omega_0)/H^1_{0,\Gamma}(\Omega_0)]\times \dots \times 
    [H^1(\Omega_m)/H^1_{0,\Gamma}(\Omega_m)]\\
    &\cong H^{1/2}(S_0)\times \dots \times H^{1/2}(S_m).
\end{align*}
Therefore, we have the injection 
\begin{align*}
\mtsd \xhookrightarrow{} H^{1/2}(S_0)\times \dots \times H^{1/2}(S_m).
\end{align*}
$\mtsn \xhookrightarrow{} H^{-1/2}(S_0)\times \dots \times H^{-1/2}(S_m)$ 
follows analogously.
\end{proof}

\begin{lemma}
\label{lem:dualityDecoupling1}
Let $\Gamma$ be a multi-screen as in Assumption~\ref{ass:ms}.
Then, for $w \in \mtsd$ and $\varphi \in \mtsn$ such that $\varphi_{|S_j} \in 
\widetilde{H}^{-1/2}(S_j) \, \forall \, j=0,\dots,m$, we have that
\begin{align}
    \ll w , \varphi \gg_{\Gamma} = \sum_{l=0}^m \langle w_{|S_l}, \varphi_{|S_l}
    \rangle_{S_l}.
\end{align}
\end{lemma}
\begin{proof}
Let us consider $w \in \mtsd$ and $\varphi \in \mtsn$. By definition
\begin{gather*}
    \ll w , \varphi \gg = \int\nolimits_{[\Gamma]} w \varphi \ d\sigma 
  = \int_{\mathbb{R}^d \backslash \overline{\Gamma}} \mathbf{p} \cdot \nabla U 
  + U \mathrm{div}(\mathbf{p}) \ d\bx,
\end{gather*}
for $U\in H^1(\R^d\setminus\overline{\Gamma})$ and $\mathbf{p} \in \mathbf{H
  }(\mathrm{div}, \R^d\setminus\overline{\Gamma})$ such that $\pi_D(U)=w$ and 
  $\pi_N(\mathbf{p}) = \varphi$.
  
For $j=0,\dots,m$, we set $U_j = U_{|\Omega_j}$ and $\mathbf{p}_j = 
\mathbf{p}_{|\Omega_j}$, and let $\mathbf{n}_j$ denote the outwards unit normal 
vector to $\Omega_j$. Then, by linearity of the integrals and Green's formula, 
we get
\begin{align}
 \int_{\mathbb{R}^d \backslash \overline{\Gamma}} \mathbf{p} \cdot \nabla U 
  + U \mathrm{div}(\mathbf{p}) \ d\bx &= \sum_{j=0}^m \, \int_{\Omega_j} 
  \mathbf{p}_j \cdot \nabla U_j + U_j \mathrm{div}(\mathbf{p}_j) \ d\bx, 
  \nonumber\\
  &= \sum_{j=0}^m \, \int_{\partial \Omega_j} (U_{j})_{|\partial\Omega_j} 
  \mathbf{n}_j \cdot (\mathbf{p}_{j})_{|\partial\Omega_j} d\sigma. 
  \label{eq:dualityaux}
\end{align}

Now, let us point out that for any $j=0,\dots, m$ we know that functions in $H^1
(\R^d\setminus \overline{\Gamma})$ and $\mathbf{p} \in \mathbf{H}(\mathrm{div}, 
\R^d\setminus\overline{\Gamma})$ do not jump across $\partial \Omega_j \setminus
\Gamma$. This allow us to simplify \eqref{eq:dualityaux} further as
\begin{align}
 \sum_{j=0}^m \, \int_{\partial \Omega_j} (U_{j})_{|\partial\Omega_j} 
 \mathbf{n}_j\cdot (\mathbf{p}_{j})_{|\partial\Omega_j} d\sigma = 
 \int_{\Gamma} \sum_{j=0}^m v_j \mu_j d\sigma,
 \label{eq:dualityRed1}
\end{align}
where $v_j = (U_j)_{|\Gamma}$ and $\mu_j = \mathbf{n}_j \cdot
(\mathbf{p}_{j})_{|\Gamma}$.

In order to see this, let us illustrate it for the case shown in 
Fig.~\ref{fig:LPMinInflated}. There $H^1(\R^d\setminus \overline{\Gamma})$ 
implies
\begin{align}
U_2 = U_0 \: \text{ on } \partial \Omega_2\setminus\Gamma, \qquad \qquad
U_1 = U_0 \: \text{ on } \partial \Omega_1\setminus\Gamma.
\label{eq:contAcrossGamma1}
\end{align}
Similarly, since $\mathbf{p} \in \mathbf{H}(\mathrm{div},\R^d\setminus\overline{
\Gamma})$ and $\mathbf{n}_2=-\mathbf{n}_0$ on $\partial \Omega_2 
\setminus \Gamma$, and $\mathbf{n}_1=\mathbf{n}_0$ on $\partial \Omega_1 
\setminus \Gamma$, we have
\begin{align}
\mathbf{n}_2 \cdot \mathbf{p}_{2} = -\mathbf{n}_0 \cdot \mathbf{p}_{0} \: 
\text{ on } \partial \Omega_2\setminus\Gamma, \qquad \qquad
\mathbf{n}_1 \cdot \mathbf{p}_{1} = -\mathbf{n}_0 \cdot \mathbf{p}_{0}  \: 
\text{ on } \partial \Omega_1\setminus\Gamma.
    \label{eq:contAcrossGamma2}
\end{align}

\begin{figure}[h]
\vspace{-1cm}
\centering
 \includegraphics[width=0.55\textwidth]{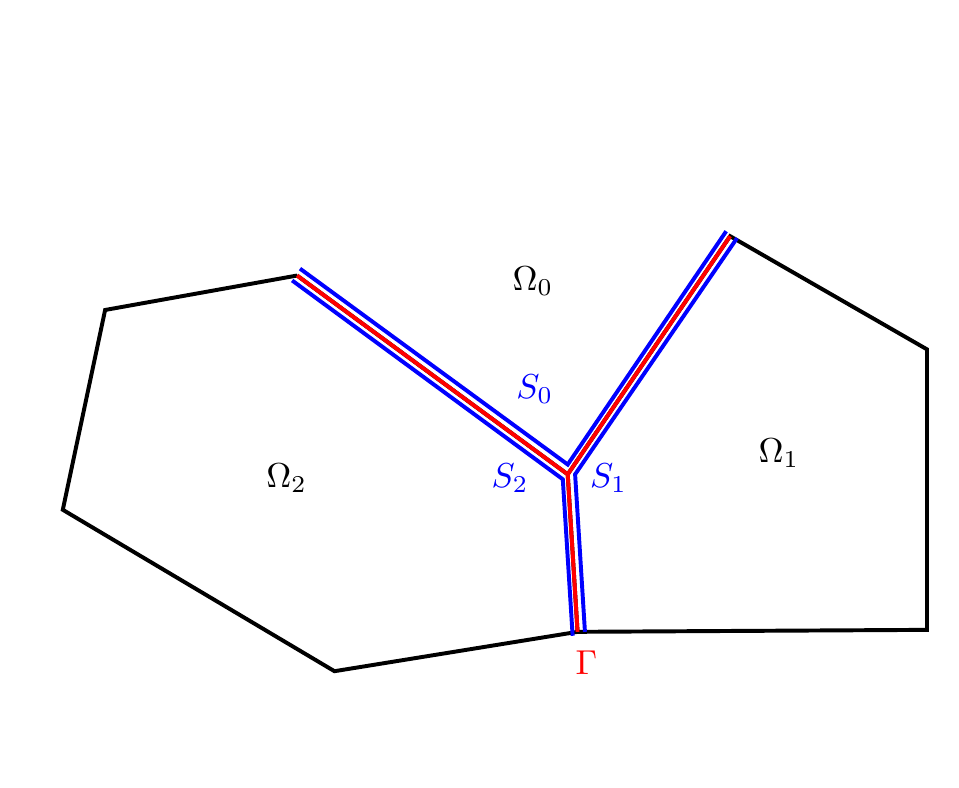}
 \caption{Example of a multi-screen $\Gamma$ with a triple junction. Here $\Gamma, S_1, S_2$ and $S_0$ overlap, so they have been drawn slightly shifted for the sake of visibility.}
 \label{fig:LPMinInflated}
\end{figure}

Then 
\begin{align}
\sum_{j=0}^2 \, \int_{\partial \Omega_j} (U_{j})_{|\partial\Omega_j}\mathbf{n 
}_j\cdot(\mathbf{p}_{j})_{|\partial\Omega_j} d\sigma = \sum_{j=0}^2 \int_{ 
\partial\Omega_j\setminus\Gamma} (U_{j})_{|\partial\Omega_j}\mathbf{n}_j \cdot 
(\mathbf{p}_{j})_{|\partial\Omega_j}d\sigma + \int_{\Gamma} \sum_{j=0}^2 v_j
\mu_j d\sigma.
\label{eq:dualityaux2}
\end{align}
Using \eqref{eq:contAcrossGamma1} and \eqref{eq:contAcrossGamma2} in 
\eqref{eq:dualityaux2}, the terms on $\partial \Omega_0 \setminus \Gamma$, 
$\partial \Omega_1 \setminus \Gamma$ and $\partial \Omega_2 \setminus \Gamma$ 
vanish, and we get
\begin{align*}
\sum_{j=0}^2 \, \int_{\partial \Omega_j} (U_{j})_{|\partial\Omega_j} \mathbf{n 
}_j \cdot (\mathbf{p}_{j})_{|\partial\Omega_j} d\sigma = \int_{\Gamma} 
\sum_{j=0}^2 v_j \mu_j d\sigma.
\end{align*}

Next, let us continue with our proof and return to \eqref{eq:dualityRed1}. We 
note that 
\begin{align*}
\mu_j = \mathbf{n}_j \cdot (\mathbf{p}_{j})_{|\partial\Omega_j} = \pi_N( 
\mathbf{p}_{|\Omega_j}) = (\pi_N( \mathbf{p}) )_{|S_j} = \varphi_{|S_j} \in 
\widetilde{H}^{-1/2}(S_j),\\
v_j = (U_{j})_{|\partial\Omega_j} = \pi_D( U_{|\Omega_j}) = (\pi_D(U))_{|S_j} = 
w_{|S_j} \in H^{1/2}(S_j).
\end{align*}
This implies that on the right hand side of \eqref{eq:dualityRed1}, we are 
allowed to split the integral over $\Gamma$ into the sum of the integrals over 
$S_j$ for $j=0,\dots,m$. Furthermore, we can write it in terms of the $H^{1/2} 
(S_j) \times\widetilde{H}^{-1/2}(S_j)$-duality pairings, i.e.
\begin{align*}
\int_{\Gamma} \sum_{j=0}^m v_j \mu_j d\sigma = \sum_{j=0}^m \int_{S_j} v_j 
\mu_j d\sigma = \sum_{j=0}^m \langle v_j, \mu_j \rangle_{S_j}.
\end{align*}
Finally, using again that $\mu_j = \varphi_{|S_j}$ and $v_j = w_{|S_j}$, we 
conclude that
\begin{equation*}
 \ll w , \varphi \gg_{\Gamma} = \sum_{j=0}^m \langle w_{|S_j}, \varphi_{|S_j} 
 \rangle_{S_j}.
\end{equation*}
\end{proof}

Analogously, one can prove:
\begin{lemma}
\label{lem:dualityDecoupling2}
Let $\Gamma$ be a multi-screen as in Assumption~\ref{ass:ms}.
Then, for $\varphi \in \mtsn$ and $w \in \mtsd$ such that $w_{|S_j} \in 
\widetilde{H}^{1/2}(S_j) \, \forall \, j=0,\dots,m$, we have that
\begin{align}
    \ll w , \varphi \gg_{\Gamma} = \sum_{l=0}^m \langle w_{|S_l}, \varphi_{|S_l}
    \rangle_{S_l}.
\end{align}
\end{lemma}

\subsection{Inverse inequalities}
In this section we will use a slightly different notation for our discrete 
multi-trace subspaces, just to allow them to be either in the primal or on the 
dual (virtual) mesh as introduced in Section~\ref{sec:CaldPreFull}. We consider
\begin{align*}
\mathbb{Z}_h^{1/2}(\Gamma) \subset \mtsd, \qquad \qquad 
\mathbb{Z}_h^{-1/2}(\Gamma) \subset \mtsn,
\end{align*}
with $\dim(\mathbb{Z}_h^{1/2}(\Gamma))<\infty$ and  
$\dim(\mathbb{Z}_h^{-1/2}(\Gamma))<\infty$.

In order to present the main results of this Section, we need to introduce some 
existing results. Let us consider the simple screens $S_j$ for $j=0,\dots,m$ and 
the following finite dimensional spaces of piecewise polynomials:
\begin{align*}
    \mathcal{S}^{0,-1}(S_{jh})\,: \, \text{ piecewise constants}, \qquad
    \mathcal{S}^{1,0}(S_{jh})\, : \, \text{ piecewise linears}.\\
\end{align*}

\begin{lemma}[{\cite[Lemma~2.8]{Heu98}}] For $j=0,\dots,m$, the following 
inverse inequalities hold:
\begin{align}
\Vert u_h\Vert_{\widetilde{H}^{1/2}(S_j)} &\leq c_1(1+\vert\log h\vert)\Vert 
u_h\Vert_{H^{1/2}(S_j)}, \qquad \forall u_h\in \mathcal{S}^{1,0}(S_{jh})\subset 
\widetilde{H}^{1/2}(S_j), \label{eq:iiSjd}\\
\Vert \varphi_h\Vert_{\widetilde{H}^{-1/2}(S_j)} &\leq c_2(1+\vert\log 
h\vert)\Vert \varphi_h\Vert_{H^{-1/2}(S_j)}, \qquad \forall \varphi_h\in 
\mathcal{S}^{0,-1}(S_{jh})\subset \widetilde{H}^{-1/2}(S_j), \label{eq:iiSjn}
\end{align}
for mesh size $h\leq1$ and with constants $c_1, c_2>0$ independent of $h$.
\end{lemma}

Next, we define two projection operators that will play a key role in their 
proofs. 
Let $\mathcal{S}^{1,0}(\check{S}_{jh})$ be the space of piecewise 
linears on the dual barycentric mesh of $S_j$, as defined in \cite{BuC07}.
We introduce the generalized $L^2$-projection $\tilde{Q}_h^j : 
L^2(S_j) \to \mathcal{S}^{1,0}(\check{S}_{jh})$ as 
\begin{align}
\langle \tilde{Q}_h^j u, \phi_{h}\rangle_{S_j} = \langle u, \phi_{h} 
\rangle_{S_j}, \qquad \forall \phi_{h} \in \mathcal{S}^{0,-1}({S}_{jh}).
\label{eq:Qjh}
\end{align}

then, following \cite[Theorems~2.1 and 2.2]{steinBook2}, one can show
\begin{lemma}[{\cite[Theorem~4.3]{HJU13}, Case A}]
Let the family of meshes $\{S_{jh}\}_{h\in \mathcal{H}} , h > 0$ of a simple screen
$S_j$ be uniformly shape-regular and locally quasi-uniform. 
Then, under certain (mild) local mesh conditions \cite[Assumption~2.1]{steinBook2}, we have that
\begin{align}
\Vert\tilde{Q}_h^j u\Vert_{H^{1/2}(S_j)} \leq c_{Qj} \Vert u \Vert_{ 
H^{1/2}(S_j)}, \qquad \forall u \in H^{1/2}(S_j), \label{eq:contQjh}
\end{align}

and that the following inf-sup condition holds 
\begin{align} \label{eq:inf-sup}
\underset{v_h \in\mathcal{S}^{1,0}(\check{S}_{jh})}{\sup} \frac{|\langle\varphi_h, v_h\rangle|}{\Vert v_h \Vert_{H^{1/2}(S_j)}} \geq \frac{1}{c_1} \Vert\varphi_h\Vert_{\widetilde{H}^{-1/2}(S_j)}\:, \qquad \forall \: \varphi_h \: \in \: \mathcal{S}^{0,-1}(S_{jh}),
\end{align}
with constants $c_1>0$ independent of $h$.
\end{lemma}

\begin{remark}
 The local mesh conditions \cite[Assumption~2.1]{steinBook2} are considered mild because they are fulfilled by a broad set of meshes used in applications, including geometrically graded meshes, algebraically 2-graded meshes and families of meshes generated by adaptive red-green algorithms \cite{GSU21}. 
\end{remark}

Let us also introduce the projection operator $\Pi_h: \mtsd \to \mathbb{Z}_h^{
-1/2}(\Gamma)\subset\mtsn$ 
\begin{align}
    \ll \Pi_h u, w_h \gg = (u, w_h)_{\mtsd}.
    \label{eq:Pih}
\end{align}

From \cite[Thm.~2.1 and 2.2]{steinBook2}, we know that under certain (mild) local mesh conditions \cite[Assumption~2.1]{steinBook2} we satisfy the discrete inf-sup condition of the duality pairing of multi-trace spaces. Moreover, these mesh conditions also guarantee the continuity of $\Pi_h$,
i.e.
\begin{align}
\Vert \Pi_h u \Vert_{\mtsn} \leq c_{\pi} \Vert  u \Vert_{\mtsd}, \qquad 
\forall u \in \mtsd. \label{eq:contPih}
\end{align}

Therefore, in order to use the continuity of both projection operators, 
we will need to satisfy the following mesh assumption:
\begin{assumption}
\label{asm:MeshAssumption}

Let $\Gamma$ be a multi-screen as in Assumption~\ref{ass:ms}.

We assume that for each $j=0,..,m$, the family of meshes $\{S_{jh}\}_{h\in \mathcal{H}} , h > 0$ 
of $S_j$ agree at the junction(s), are uniformly shape-regular, locally quasi-uniform and satisfy
the (mild) local mesh conditions from \cite[Assumption~2.1]{steinBook2}.
\end{assumption}

\begin{lemma}[Inverse inequality in $\mtsn$]
\label{lem:invIneqNeu}
Let $\Gamma$ be a multi-screen as in Assumption~\ref{ass:ms} and
$\mathbb{Z}_h^{1/2}(\Gamma)$ such that for all $v_h \in 
\mathbb{Z}_h^{1/2}(\Gamma)$ it holds that $(v_h)_{|S_j} \in \widetilde{H}^{1/2}
(S_j)$. Then, we have that for all $\varphi_h \in \mathbb{Z}_h^{-1/2}(\Gamma)$ 
\begin{equation}
\Vert \varphi_h \Vert_{\mtsn} \leq \widetilde{C}_N (1 + \vert \log h \vert ) 
\left( \sum_{j=0}^m \Vert \varphi_h \Vert_{H^{-1/2}(S_j)}^2 \right)^{1/2}, 
\label{eq:invIneqNeu}
\end{equation}
with mesh size $h\leq1$, and $\widetilde{C}_N$>0 independent of $h$.
\end{lemma}
\begin{proof}
By definition of dual norm and Lemma~\ref{lem:dualityDecoupling1}, we get
\begin{align}
\Vert \varphi_h \Vert_{\mtsn} &= \underset{v\in\mtsd\setminus\lbrace0\rbrace} 
{\sup} \dfrac{\vert \ll v,\varphi_h\gg\vert}{\Vert v \Vert_{\mtsd}}
= \underset{v\in\mtsd\setminus\lbrace0\rbrace}{\sup}\dfrac{\vert \sum_{j=0}^m 
\langle v_j, \varphi_{hj} \rangle_{S_j} \vert}{\Vert v \Vert_{\mtsd}},
\label{eq:iibound1}
\end{align}
where we have set $v_j = v_{|S_j}$ and $\varphi_{hj} = (\varphi_h)_{|S_j}$. 
Then, let us consider the index set $\mathcal{J}$ of all the indices $0\leq k 
\leq m$ such that $v_k \neq 0$. Thus, we have that 
\begin{align}
\dfrac{\vert \sum_{j=0}^m \langle v_j, \varphi_{hj} \rangle_{S_j} \vert}{\Vert v
\Vert_{\mtsd}} = \dfrac{\vert \sum_{j\in\mathcal{J}} \langle v_j, \varphi_{hj} 
\rangle_{S_j} \vert}{\Vert v \Vert_{\mtsd}}.
\label{eq:aux1}
\end{align}

Next, we derive two inequalities that will help us proceed. First, using the 
embedding from Lemma~\ref{lem:MTinjections} and then Young's inequality 
$m$-times, we can obtain
\begin{align}
\Vert v \Vert^2_{\mtsd} \geq \sum_{j\in\mathcal{J}} \Vert v_j \Vert^2_{ 
H^{1/2}(S_j)} 
\geq \dfrac{1}{\vert \mathcal{J} \vert} \left( \sum_{j\in\mathcal{J}} \Vert v_j 
\Vert_{H^{1/2}(S_j)} \right)^2
\label{eq:aux2}
\end{align}
where $\vert \mathcal{J} \vert$ is the size of the index set $\mathcal{J}$. 
Second, we remark that for a sum $\sum_{i=1}^{m^*} a_i$ with all coefficients 
$a_i >0$ and $m^*\in \mathbb{N}$, we have that $\dfrac{1}{\sum_{i=1}^{m^*} a_i} \leq \dfrac{1}{a_k}$ for
all $k=1,\dots m^*$. Hence, 
\begin{align}
\dfrac{m^*}{\sum_{i=1}^{m^*} a_i} \leq \sum_{k=1}^{m^*} \dfrac{1}{a_k}
    \label{eq:aux3}
\end{align}

Plugging these in \eqref{eq:iibound1} gives
\begin{align}
\Vert \varphi_h \Vert_{\mtsn} &\overset{\eqref{eq:aux1}}{=} \underset{v\in\mtsd 
\setminus\lbrace0\rbrace}{\sup}\dfrac{\vert \sum_{j\in\mathcal{J}} \langle v_j, 
\varphi_{hj} \rangle_{S_j} \vert}{\Vert v \Vert_{\mtsd}} 
\quad\quad\quad \overset{\eqref{eq:aux2}}{\leq} \vert\mathcal{J}\vert^{1/2} 
\underset{v\in\mtsd\setminus\lbrace0\rbrace}{\sup}\dfrac{\vert\sum_{j\in 
\mathcal{J}} \langle v_j,\varphi_{hj}\rangle_{S_j}\vert}{\sum_{k\in\mathcal{J}}
\Vert v_k \Vert_{H^{1/2} (S_k)}} \nonumber \\
\:&\overset{\eqref{eq:aux3}}{\leq} \dfrac{1}{\vert \mathcal{J}\vert^{1/2}} 
\underset{v\in\mtsd\setminus\lbrace0\rbrace}{\sup}\vert \sum_{j\in\mathcal{J}} 
\dfrac{\langle v_j, \varphi_{hj}\rangle_{S_j}}{\Vert v_j \Vert_{H^{1/2} (S_j)}} 
\vert 
\leq \dfrac{1}{\vert \mathcal{J}\vert^{1/2}} \underset{v\in\mtsd\setminus 
\lbrace0\rbrace}{\sup} \sum_{j\in\mathcal{J}} \dfrac{\vert\langle v_j,
\varphi_{hj}\rangle_{S_j}\vert}{\Vert v_j \Vert_{H^{1/2}(S_j)}} \nonumber \\
\:&\leq \dfrac{1}{\vert \mathcal{J}\vert^{1/2}} \sum_{j\in\mathcal{J}} \:
\underset{w_j\in H^{1/2}(S_j)\setminus\lbrace0\rbrace}{\sup} 
\dfrac{\vert\langle w_j, \varphi_{hj}\rangle_{S_j}\vert}{\Vert w_j 
\Vert_{H^{1/2}(S_j)}} 
= \dfrac{1}{\vert \mathcal{J}\vert^{1/2}} \sum_{j=0}^{m} \:
\underset{w_j\in H^{1/2}(S_j)\setminus\lbrace0\rbrace}{\sup} \dfrac{\vert 
\langle w_j, \varphi_{hj}\rangle_{S_j}\vert}{\Vert w_j \Vert_{H^{1/2} (S_j)}}.
\end{align}

Then, using $\tilde{Q}_h^j$ from \eqref{eq:Qjh} and its continuity 
\eqref{eq:contQjh}, we get
\begin{align}
\Vert \varphi_h \Vert_{\mtsn} &\leq \dfrac{1}{\vert \mathcal{J}\vert^{1/2}} 
\sum_{j=0}^{m} \: \underset{w_j\in H^{1/2}(S_j)\setminus\lbrace0\rbrace}{\sup} 
\dfrac{\vert\langle \tilde{Q}_h^j w_j, \varphi_{hj}\rangle_{S_j}\vert}{\Vert w_j
\Vert_{H^{1/2} (S_j)}} \nonumber 
\leq \dfrac{1}{\vert \mathcal{J}\vert^{1/2}} 
\sum_{j=0}^{m} \: c_{Qj} \underset{w_j\in H^{1/2}(S_j)\setminus\lbrace0\rbrace} 
{\sup} \dfrac{\vert\langle \tilde{Q}_h^j w_j, \varphi_{hj}\rangle_{S_j}\vert} 
{\Vert  \tilde{Q}_h^j w_j \Vert_{H^{1/2} (S_j)}} \nonumber\\
&\leq \dfrac{1}{\vert \mathcal{J}\vert^{1/2}} 
\sum_{j=0}^{m} \: c_{Qj} \underset{w_{hj}\in \mathcal{S}^{0,1}(S_{jh}) 
\setminus\lbrace0\rbrace}{\sup} \dfrac{\vert\langle w_{hj}, \varphi_{hj} 
\rangle_{S_j}\vert} {\Vert  w_{hj} \Vert_{H^{1/2} (S_j)}}. 
\label{eq:iibound2}
\end{align}

Using Cauchy-Schwarz, we obtain
\begin{align}
\Vert \varphi_h \Vert_{\mtsn} &\leq \dfrac{1}{\vert \mathcal{J}\vert^{1/2}} \sum_{j=0}^{m} \: c_{Qj} \underset{w_{hj}\in \mathcal{S}^{0,1}(S_{jh}) 
\setminus\lbrace0\rbrace}{\sup} \dfrac{\Vert  w_{hj} \Vert_{H^{1/2} (S_j)}
\Vert  \varphi_{hj} \Vert_{\widetilde{H}^{-1/2} (S_j)}} {\Vert  w_{hj} \Vert_{H^{1/2} (S_j)}}\leq \dfrac{1}{\vert \mathcal{J}\vert^{1/2}} \sum_{j=0}^{m} \: c_{Qj} \Vert  \varphi_{hj} \Vert_{\widetilde{H}^{-1/2} (S_j)}
\end{align}

Using the inverse inequality on simple screens \eqref{eq:iiSjn}, we can further 
bound this as
\begin{align}
\Vert \varphi_h \Vert_{\mtsn} &\leq c_N(1 + \vert \log h\vert) \sum_{j=0}^{m} \: \Vert  \varphi_{hj} \Vert_{ H^{-1/2} (S_j)},
\end{align}
with $c_N := \dfrac{c_2 \max(c_{Qj})}{\vert \mathcal{J}\vert^{1/2}}$.

For convenience, we can work the estimate a bit further 
\begin{align}
\Vert \varphi_h \Vert_{\mtsn}^2 &\leq c_N^2 (1 + \vert \log h\vert)^2
\left(\sum_{j=0}^{m}\: \Vert\varphi_{hj}\Vert_{H^{-1/2}(S_j)}\right)^2 
\: \overset{\eqref{eq:aux2}}{\leq}\vert \mathcal{J}\vert c_N^2 (1 + \vert\log 
h\vert)^2 \sum_{j=0}^{m} \: \Vert  \varphi_{hj} \Vert_{H^{-1/2} (S_j)}^2,
\end{align}
which gives \eqref{eq:invIneqNeu} with $\tilde{C}_N = c_2\max(c_{Qj})$.
\end{proof}

\begin{lemma}[Inverse inequality in $\mtsd$]
\label{lem:invIneqDir}
Let $\Gamma$ be a multi-screen as in Assumption~\ref{ass:ms} and 
$\mathbb{Z}_h^{-1/2}(\Gamma)$ such that for all $\mu_h \in 
\mathbb{Z}_h^{-1/2}(\Gamma)$ it holds that $(\mu_h)_{|S_j} \in \widetilde{H 
}^{-1/2}(S_j)$. Then, we have that for all $u_h \in \mathbb{Z}_h^{1/2}(\Gamma)$
\begin{equation}
\Vert u_h \Vert_{\mtsd} \leq \widetilde{C}_D (1 + \vert \log h \vert ) \left( 
\sum_{j=0}^m \Vert u_h \Vert_{H^{1/2}(S_j)}^2 \right)^{1/2}, 
\end{equation}
with mesh size $h\leq1$, and $\widetilde{C}_D$>0 independent of $h$.
\end{lemma}
\begin{proof}
By definition and continuity of $\Pi_h$ (\emph{c.f.} \eqref{eq:contPih}), we 
have that
\begin{align}
\Vert u_h\Vert_{\mtsd} = \dfrac{\vert(u_h, u_h)_{\mtsd}\vert}{\Vert u_h\Vert_{ 
\mtsd}} = \dfrac{\vert\ll u_h,\Pi_h u_h\gg\vert}{\Vert u_h\Vert_{\mtsd}} \leq 
c_{\pi} \dfrac{\vert\ll u_h,\Pi_h u_h\gg\vert}{\Vert\Pi_h u_h\Vert_{\mtsn}}.
\label{eq:iiboundN}
\end{align}
For convenience, set $\varphi_h=\Pi_h u_h \in\mathbb{Z}_h^{-1/2}(\Gamma)$, and 
$\varphi_{hj} =(\varphi_h)_{|S_j}$. By assumption, we have that $\varphi_{hj}
\in \widetilde{H}^{ -1/2}(S_j)$. and thus we can apply 
Lemma~\ref{lem:dualityDecoupling1} and split the duality pairing. 

In the following, we will proceed in analogy to what we did in the proof of 
Lemma~\ref{lem:invIneqNeu}. For this, let us introduce the index set 
$\mathcal{J}$ of all the indices $0\leq k\leq m$ such that $\varphi_{hj}\neq0$. 
Moreover, similarly to \eqref{eq:aux2}, one can derive
\begin{align}
\Vert \mu \Vert^2_{\mtsn} \geq \sum_{j\in\mathcal{J}} \Vert \mu_j \Vert^2_{ 
H^{-1/2}(S_j)} \geq\dfrac{1}{\vert\mathcal{J}\vert} \left(\sum_{j\in\mathcal{J}}
\Vert \mu_j \Vert_{H^{-1/2}(S_j)} \right)^2.
\label{eq:aux4}
\end{align}

With this, we can bound \eqref{eq:iiboundN} further as follows
\begin{align}
\Vert u_h \Vert_{\mtsd} &\leq c_{\pi} \vert\mathcal{J}\vert^{1/2} 
\dfrac{\vert \sum_{j\in\mathcal{J}} \langle u_{hj},\varphi_{hj}\rangle_{S_j} 
\vert}{\sum_{k\in\mathcal{J}}\Vert\varphi_{hk}\Vert_{H^{-1/2}(S_k)}}
\overset{\eqref{eq:aux3}}{\leq} \dfrac{c_{\pi}}{\vert\mathcal{J}\vert^{1/2}} 
\sum_{j\in\mathcal{J}} \dfrac{\vert \langle u_{hj},\varphi_{hj}\rangle_{S_j} 
\vert}{\Vert\varphi_{hj}\Vert_{H^{-1/2}(S_j)}}\nonumber\\&\leq 
\dfrac{c_{\pi}}{\vert\mathcal{J} \vert^{1/2}} \sum_{j\in\mathcal{J}} \dfrac{ 
\Vert u_{hj}\Vert_{H^{1/2}(S_j)}\Vert\varphi_{hj}\Vert_{\widetilde{H}^{-1/2}
(S_j)}}{\Vert\varphi_{hj}\Vert_{H^{-1/2}(S_j)}}.
\end{align}
Now, we use the inverse inequality \eqref{eq:iiSjn} on $S_j$ and get 
\begin{align}
\Vert u_h \Vert_{\mtsd} \leq \dfrac{c_{\pi}}{\vert\mathcal{J} \vert^{1/2}}
c_2 (1+\log h) \sum_{j\in\mathcal{J}} \Vert u_{hj}\Vert_{H^{1/2}(S_j)}\nonumber
\leq \dfrac{c_{\pi}}{\vert\mathcal{J} \vert^{1/2}}
c_2 (1+\log h) \sum_{j=0}^m \Vert u_{hj}\Vert_{H^{1/2}(S_j)}.
\end{align}
Finally, just as in Lemma~\ref{lem:invIneqNeu}, we point out that from this one 
can derive
\begin{align}
\Vert u_h \Vert_{\mtsd} &\leq \tilde{C}_D (1+\log h) \left(\sum_{j=0}^m \Vert 
u_{hj}\Vert_{H^{1/2}(S_j)}^2\right)^{1/2}.
\end{align}
\end{proof}

\subsection{Norm equivalences}
\label{app:NormEquiv}

Recall that 
\begin{equation*}
    \Vert u \Vert_{\jspd} = \inf_{x \in \stsd} \Vert u + x \Vert_{\mtsd},
\end{equation*}
and that $\mathcal{S}^{1,0}(\meshsymb_h)$ is the space spanned by piecewise 
linear ``continuous'' functions on (the virtual mesh) $\meshsymb_h$. We define 
the norm
\begin{equation}
    \Vert u \Vert_{\jspd,D} := \inf_{x_h \in \stsd\cap
    \mathcal{S}^{1,0}(\meshsymb_h)} \Vert u + x_h \Vert_{\mtsd},
\end{equation}
and study its relation with the continuous jump norm.

\begin{lemma}
\label{lem:normEquivalenceD}
 Assume that there exists an operator $\OR_h^+ \, : \, \mtsd \to \mathcal{S}^{1,0}(\meshsymb_h)$ such that
\begin{itemize}
 \item[(i)] $\OR_h^+$ is a h-uniformly bounded projection,
 \item[(ii)] $\OR_h^+ ( \stsd ) \subseteq  \stsd\cap\mathcal{S}^{1,0}(\meshsymb_h). $
\end{itemize}

Then $ \Vert \cdot \Vert_{\jspd} $ and $ \Vert \cdot \Vert_{\jspd,D} $ are equivalent norms
\end{lemma}
\begin{proof}
 By definition, we have
\begin{equation}
    \Vert u \Vert_{\jspd}  \leq \Vert u \Vert_{\jspd,D} 
    \label{eq:normIneqD1}
\end{equation}

Choose $u_h = \OR_h^+ u$ where $u\in \jspd \subset \mtsd$.
Then we have that 
\begin{equation*}
    \Vert u_h \Vert_{\jspd,D} = \inf_{x_h \in \stsd\cap
    \mathcal{S}^{1,0}(\meshsymb_h)} \Vert u_h + x_h \Vert_{\mtsd}.
\end{equation*}
Note that, by surjectivity of $\OR_h^+$ and property (ii), there exists an $x \in \stsd$ such that 
$x_h = \OR_h^+ x$. This allows us to write
\begin{equation*}
\Vert u_h \Vert_{\jspd,D} = \inf_{x_h \in \stsd} \Vert \OR_h^+(u+x) \Vert_{\mtsd}.
\end{equation*}
Since $\OR_h^+$ is  continuous, we further get
\begin{equation}
\Vert u_h \Vert_{\jspd,D} \leq \Vert \OR_h^+ \Vert\inf_{x_h \in \stsd} 
\Vert u+x \Vert_{\mtsd} = \Vert \OR_h^+ \Vert \Vert u_h \Vert_{\jspd}
\label{eq:normEstimateD1}
\end{equation}
for all $v_h \in \jspd\cap \mathcal{S}^{1,0}(\meshsymb_h)$.

From this, we conclude that the two norms are equivalent. 
Furthermore, the fact that $\OR_h^+$ is $h$-uniformly bounded guarantees that the 
related constants are $h$-independent.
\end{proof}

Similarly, one shows that for the norm
\begin{equation}
    \Vert \mu \Vert_{\jspn,D} := \inf_{\xi_h \in \stsn\cap
    \mathcal{S}^{0,-1}(\meshsymb_h)} \Vert \mu + \xi_h \Vert_{\mtsn},
\end{equation}
we have 
\begin{lemma}
\label{lem:normEquivalenceN}
 Assume that there exists an operator $\OR_h^- \, : \, \mtsn \to \mathcal{S}^{0,-1}(\meshsymb_h)$ such that
\begin{itemize}
 \item[(i)] $\OR_h^-$ is a h-uniformly bounded projection,
 \item[(ii)] $\OR_h^- ( \stsn ) \subseteq  \stsn\cap\mathcal{S}^{0,-1}(\meshsymb_h). $
\end{itemize}

Then, there exists $C_N>0$ independent of $h$ such that
\begin{equation}
    \Vert \varphi_h \Vert_{\jspn}  \leq \Vert \varphi_h \Vert_{\jspn,D} \leq 
c_N \Vert \varphi_h \Vert_{\jspn}
\label{eq:normEquivalenceN}
\end{equation}
for all $\varphi_h \in \jspn\cap \mathcal{S}^{0,-1}(\meshsymb_h)$.
\end{lemma}

\subsection{Condition number estimates for quotient space discretisations}
\label{app:CondNumEstimates}

Following the policy of operator preconditioning, we need to bound the spectral
condition number $\kappa_{sp}(\VM^{-1} \VB_{h}\VM^{-T} \VA_{h})$ for the pairs 
$(\VA_h, \VB_h)=(\VW_{\kappa,h}, \VB^{\OW}_{\kappa,h})$, and $(\VA_h, \VB_h) 
=(\VV_{\kappa,h}, \VB^{\OV}_{\kappa,h})$. In other words, we are assuming that we 
have that
\begin{itemize}
    \item the Galerkin matrix $\VB_h$ arises from a continuous sesquilinear form $\bb$ 
    that satisfies a discrete inf-sup condition on the whole space $\besmtd$;
    \item the Galerkin matrix $\VM_h$ arises from a continuous sesquilinear form $\bm$ 
    that satisfies a discrete inf-sup condition on the whole space $\besmtp$;
    \item the Galerkin matrix $\VA_h$ has a non-trivial nullspace $N_h(\Gamma) 
    := \ker \OA_h$. Moreover, 
    its corresponding sesquilinear form $\ba$ satisfies a discrete inf-sup 
    condition only on $\besmtp/N_h(\Gamma)$.
\end{itemize}

Since $\VA_h$ is either $\VW_{\kappa,h}$ or $\VV_{\kappa,h}$, we will directly 
use that $N_h(\Gamma) = X_h(\Gamma)$ to avoid unnecessary extra notation.

As usual, this entails bounding $\lambda_{\max}
:= \lambda_{\max}(\VM^{-1}\VB_{h}\VM^{-T}\VA_{h})$ from above and $\lambda_{\min} := 
\lambda_{\min}(\VM^{-1}\VB_{h}\VM^{-T}\VA_{h})$ from below. In order to write these bounds, 
we need to introduce some notation first.

Let $\OA_h \,:\, \besmtp\to\besmtp^\prime$, $\OB_h \,:\, \besmtd\to\besmtd^\prime$ and 
$\OM_h \,:\, \besmtp\to\besmtd^\prime$ be the bounded linear operators associated to 
the sesquilinear forms $\ba$, $\bb$ and $\bm$, respectively.

For $\lambda_{\max}$ we proceed in the classical way and arrive to 
\begin{align}
\lambda_{\max} &\leq 
\Vert \OM^{-1} \Vert^2\Vert \OB \Vert \Vert \OA 
\nonumber \\
&= \alpha_{\OM}^2
\Vert \OB \Vert \Vert \OA \Vert, 
\label{eq:boundLmax}
\end{align}
where 
\begin{align*}\small
\alpha_{\OM} := \inf_{v_h\in \besmtp\setminus\{0\}} \dfrac{\Vert\OM_h v_h\Vert_{\smtd^\prime}}{\Vert v_h \Vert_{\smtp}}.
\end{align*}

For $\lambda_{\min}$ we have to take a slightly different approach since we need to restrict $\OA_h$ 
to the space where its corresponding bilinear form $\ba$ satisfies a discrete inf-sup 
condition. Moreover, we need to establish when the discrete inf-sup condition will 
bound the smallest eigenvalue. We study this in the next Lemma.

\begin{lemma}[Discrete inf-sup constant in the quotient space norm]
Let $\ba$ be a continuous bilinear form on $\smtp\times \smtp$. Let $X(\Gamma)
\subseteq \smtp$ be both the left and right nullspace of $\ba$. 

 Let $\besmtp$ be a finite dimensional subspace of $\smtp$ and $\OA_h \,:\, 
 \besmtp\to\besmtp^\prime$ the bounded linear operator associated to $\ba$. 
 We assume that $\besmtp$ is nullspace conforming to $\ba$ in the sense that 
 $X_h(\Gamma) := \ker \OA_h = \ker \OA'_h$ is a linear subspace of $X(\Gamma)\cap 
 \besmtp\subseteq X(\Gamma)$. 

If $\ba$ satisfies a discrete inf-sup condition in the quotient space $\besmtp/X_h(\Gamma)
\times \besmtp/X_h(\Gamma)$ with constant $\alpha_{\ba} > 0$ and if the norms on $\smtp/X(\Gamma)$ and $\smtp/X_h(\Gamma)$ are equivalent, then 
\begin{equation}
    \sup_{v_h \in \besmtp\setminus \{0\}} \frac{\ba(u_h,v_h)}{ \|v_h\|_{\smtp}} \geq \tilde{\alpha}_{\ba} \|u_h\|_{\smtp/X(\Gamma)}.
    \label{eq:inf-supDiscreteQS}
\end{equation}
\end{lemma}
\begin{proof}
We note that due to the kernel, the supremum is not attained at $v_h \in X_h(\Gamma)$, 
and thus
\begin{align}
    \sup_{v_h \in \besmtp\setminus \{0\}} \frac{\ba(u_h,v_h)}{ \|v_h\|_{\smtp}} \geq \sup_{\tilde{v}_h \in X_h(\Gamma)^\perp\setminus \{0\}} \frac{\ba(u_h,\tilde{v}_h)}{ \|\tilde{v}_h\|_{\smtp}} 
    \label{eq:inf-supDiscreteQS2}
\end{align}
Let $\mathbb{Q}_h: \besmtp \to X_h(\Gamma)^\perp$ be an orthogonal projection with 
respect to the $\smtp$-inner product. Using that $\tilde{v}_h = \mathbb{Q}_h \tilde{v}_h$, and the identification of $X_h(\Gamma)^\perp$ with $\besmtp/X_h(\Gamma)$, we have that by property of quotient spaces \cite[Eq.~(4)]{CWHM17}
\begin{align}
    \sup_{\tilde{v}_h \in X_h(\Gamma)^\perp\setminus \{0\}} \frac{\ba(u_h,\tilde{v}_h)}{ \|\tilde{v}_h\|_{\smtp}} &= 
        \sup_{\tilde{v}_h \in X_h(\Gamma)^\perp\setminus \{0\}} \frac{\ba(u_h,\tilde{v}_h)}{ \|\mathbb{Q}_h \tilde{v}_h\|_{\smtp}} 
        =
            \sup_{\tilde{v}_h \in \besmtp/X_h(\Gamma)\setminus \{0\}} \frac{\ba(u_h,\tilde{v}_h)}{ \|\tilde{v}_h\|_{\smtp/X_h(\Gamma)}}. 
    \label{eq:inf-supDiscreteQS3}
\end{align}
Finally, by norm equivalence, this becomes the discrete inf-sup condition, i.e.
\begin{align}
    \sup_{v_h \in \besmtp\setminus \{0\}} \frac{\ba(u_h,v_h)}{ \|v_h\|_{\smtp}} &\geq       c \,  \sup_{\tilde{v}_h \in \besmtp/X_h(\Gamma)\setminus \{0\}} \frac{\ba(u_h,\tilde{v}_h)}{ \|\tilde{v}_h\|_{\smtp/X(\Gamma)}} \\
    &\geq       c \alpha_{\ba} |u_h\|_{\smtp/X(\Gamma)},    
    \label{eq:inf-supDiscreteQS4}
\end{align}

and therefore the required result follows with $\tilde{\alpha}_{\ba} = c \alpha_{\ba}$, 
where $c$ is the constant from the norm-equivalence.
\end{proof}

Using all the above, we can bound $\lambda_{\min}$ as follows
\begin{align}
\lambda_{\min} &= \inf_{u_h\in \besmtp/X_h(\Gamma)\setminus\{0\}} 
\dfrac{\Vert\OM_h^{-1}\OB_{h}\OM^{-*}_h{\OA_h} u_h\Vert_{
\smtp^\prime}}{\Vert u_h \Vert_{\smtp}} 
\geq \alpha_{\OM^{-1}}^2 \alpha_{\OB} \alpha_{\ba},
\label{eq:boundLmin}
\end{align}
with 
\begin{align*}
\alpha_{\OM^{-1}} := \inf_{\mu_h\in \besmtd^\prime\setminus\{0\}} \dfrac{\Vert\OM_h^{-1}\mu_h\Vert_{\smtp}}{\Vert \mu_h \Vert_{\smtd^\prime}} = 
\Vert \OM \Vert^{-1}, 
\qquad \qquad
\alpha_{\OB}:= \inf_{\xi_h\in \besmtd\setminus\{0\}} \dfrac{\Vert\OB_{h}\xi_h\Vert_{\smtd^\prime}}{\Vert \xi_h \Vert_{\smtd}}.
\end{align*}

Finally, we combine \eqref{eq:boundLmax} and \eqref{eq:boundLmin} and get
\begin{align}
    \kappa_{sp}(\VM^{-1} \VB_{h}\VM^{-T} \VA_{h}) \leq \dfrac{\alpha_{\OM}^2
\Vert \OB \Vert \Vert \OA \Vert}{\Vert \OM \Vert^2 \alpha_{\OB} \tilde{\alpha}_{\ba}}.
\end{align}

\begin{remark}
Note that the nullspace conformity requirement forces us to use meshes that agree on the \emph{front} and \emph{back} of the structure. In practice this does not pose a major limitation because in a typical usage scenario the simple screens $\partial \Omega_j \cap \Gamma$ are built by fusing together two or more unique meshes for the interfaces $\partial \Omega_i \cup \partial \Omega_j$. The interface meshes are used both as front and back and thus necessarily agree.
\end{remark}

\begin{remark}
For the discrete quotient norm to be bounded from below by the continuous quotient norm, it suffices that $\besstp \subseteq X(\Gamma) \cap \besmtp$; equality is not required. This opens the door to reduction schemes for the multi-trace space variational formulation. A reduction scheme is a choice $\besmtp^\circ \subseteq \besmtp \subseteq \smtp$ with corresponding discrete left/right nullspace $\besstp^\circ$ such that (i) $\besstp^\circ \subseteq X(\Gamma)$ and (ii) $\besmtp^\circ / \besstp^\circ = \besmtp / \besstp$. Such a choice leads on one hand to approximations in the jump space of equal quality and on the other hand does not preclude the construction of efficient operator preconditioners.
\end{remark}

\end{document}